\documentclass{amsart}
\usepackage{amsmath}
\usepackage{amsfonts}
\usepackage{amssymb}
\usepackage{graphicx}
\usepackage{MnSymbol}
\usepackage{tikz}
\usetikzlibrary{positioning}
\usepackage{array}
\usepackage{enumerate}
\usepackage{float}
\usepackage{array}
\newcolumntype{P}[1]{>{\centering\arraybackslash}p{#1}}

\usepackage[pagewise]{lineno}

\usepackage[ruled,vlined]{algorithm2e}
\usepackage{colortbl}

\usepackage{wrapfig}
\usepackage{booktabs}
\usepackage{subcaption}
\usepackage{lipsum}
\usepackage{epstopdf}

\newtheorem{theorem}{Theorem}

\newtheorem{conjecture}{Conjecture}
\newtheorem{corollary}{Corollary}

\newtheorem{definition}{Definition}
\newtheorem{example}{Example}

\newtheorem{proposition}{Proposition}
\newtheorem{remark}{Remark}

\def\frac#1#2{{\begingroup #1\endgroup\over #2}}
\newcommand\restr[2]{{
  \left.\kern-\nulldelimiterspace 
  #1 
  \right|_{#2} 
  }}

\begin{document}

\title{Dessins d'Enfants of Trigonal Curves}

\author{MEHMET EMIN AKTA\c{S}}
\address{Department of Mathematics, Florida State University, Tallahassee, Florida 32306}
\email{maktas@math.fsu.edu}
\subjclass[2010]{Primary 14H57, 14H45; Secondary 14H30}



\keywords{Trigonal curves, dessins d'enfants, $j$-invariant}


\begin{abstract}    
In this paper, we focus on properties of dessins d'enfants associated to trigonal curves. Degtyarev studied dessins d'enfants to compute braid monodromies and fundamental groups of trigonal curves using their combinatorial data. We first classify all possible combinatorial data that can occur for trigonal curves of low degree, as well as bounds on the number of possibilities for all degree. We also study deformations of trigonal curves and corresponding deformations of their dessins. Of special interest to Degtyarev was the case when the dessins are maximal. We give a sufficient condition for a trigonal curve to be deformable to one that is maximal.

\end{abstract}

\maketitle


\section{Introduction}
A dessin d'enfant, means ``child's drawing" in French, is a type of an embedded graph satisfies the following two structures:
\begin{enumerate}
\item there is a cyclic ordering on the edges of each vertex,
\item each vertex is assigned one of two colors, black or white, and the two end points of each edge have different colors.
\end{enumerate}

It was first introduced by Felix Klein \cite{klein1879ueber}. After a century, it was rediscovered and named by A. Grothendieck \cite{groth} for studying rational maps with three critical points. 

Dessins can arise from a finite covering $j:S \rightarrow P^1$, where $P^1$ is a Riemann sphere and $S$ is a Riemann surface, whose ramification locus includes $0,1, \infty$. For such a covering, one constructs the dessin as follows:

\begin{definition}\em
The \textit{dessin associated to the covering} $j$ is the embedded graph $\Gamma=j^{-1}([0,1]) \in P^1$ where $j^{-1}(0)$ are $\bullet$- vertices and $j^{-1}(1)$ are $\circ$- vertices.   
\end{definition}

Recently, A. Degtyarev associates dessins with trigonal curves to study the topology of these curve complements \cite{degtyarev2012topology}. He defines an orientation preserving covering $j_C: P^1 \rightarrow P^1$ for a given trigonal curve $C$ to define the dessin associated to $C$. Then he uses the combinatorial data of the dessins to compute the braid monodromy of the curve complements which can be used to compute their fundamental group \cite{degtyarev2009plane,degtyarev2009zariski, degtyarev2010Dessinplane,degtyarev2010plane,degtyarev2011hurwitz, degtyarev2011fundamental} and Alexander polynomial \cite{degt_alex_module, degt_2alex_module}. In other words, he defines the following two maps

$$
\mathfrak{T} \xrightarrow[]{\Omega_1} \mathfrak{D} \xrightarrow[]{\Omega_2} \mathfrak{G}
$$

where $\mathfrak{T}$ is the set of all trigonal curves, $\mathfrak{D}$ is the set of all dessins d'enfants, $\mathfrak{G}$ is the set of fundamental groups of the trigonal curve complements, $\Omega_1$ sends a trigonal curve to its dessin d'enfant and $\Omega_2$ sends a dessin d'enfant to the fundamental group of the corresponding trigonal curve complements. He defines $\Omega_2$ only when the dessin is maximal (will be defined later).

The map $\Omega_1$ is neither one-to-one, since two different trigonal curves can have the same dessin up to the graph isomorphism, nor onto, since not all dessin can be constructed using trigonal curves.

\textbf{Main results.} In this paper, we focus on the map $\Omega_1$. We first restrict $\Omega_1$ to completely reducible trigonal curves (CRTC). A trigonal curve is completely reducible if its affine part is defined by $(y-y_1)(y-y_2)(y-y_3)=0$ where $y_i \in \mathbb{C}[x]$ for all $i\in \{1,2,3\}$. We study the \textit{combinatorial data} (Section 4), \textit{deformations} (Section 5) and \textit{maximality} (Section 6) of the dessins of the completely reducible trigonal curves. 

We first fix the type of the combinatorial data as follows: 
\begin{definition}\em
Two dessins are in the same combinatorial type if both have the same number of \textit{regions} and \textit{edges in each region}.
\end{definition}
This definition is motivated from Degtyarev's method in using the dessins to study topology. We first find an upper bound for the number of different combinatorial types for a fixed degree:

\begin{theorem}
\em Let $C$ be a completely reducible trigonal curve with the maximal degree $n$, i.e. $n$ is the maximum of the degrees of its three components. Let $Comb(n)$ be the set of all combinatorial types of dessins of $C$. Then, we have 
$$
|Comb(n)|\leq \frac{p(n)(p(n)+1)(p(n)+2)}{6} - K(n)
$$
with
$$
K(n)=\sum_{r=1}^{\lfloor \frac{n+1}{3} \rfloor}p(n,r)\sum_{j=r}^{\lfloor \frac{n+1-r}{2} \rfloor,}p(n,j)\sum_{i=j}^{n-j-r+1} p(n,i).
$$
where $p(n)$ is the number of integer partitions of $n$, $p(n,m)$ is the number of integer partitions of $n$ with length $m$.
\end{theorem}  

Furthermore, we classify all possible combinatorial types of dessins up to the maximal degree 4. We prove that for the maximal degrees $1,2,3,4$, there are $1,3,8,22$ combinatorial types respectively (see Section \ref{spec_deg} for details). 

Beside the combinatorial types of dessins, we study the deformations of dessins d'enfants. We first prove that a dessin d'enfant can have four different elementary moves. We then show that every dessin d'enfant is a deformation of a simple dessin d'enfant (a dessins d'enfant is simple if it has $\bullet$- vertices of degree 3, $\circ$- vertices of degree 2 and no monochrome vertex, the vertices arise from the ramification points other than $0,1, \infty$). We prove that simple dessins d'enfants form open subsets in the deformation space of dessins d'enfants. We also find an approximation for the the number of simple dessins d'enfants for a fixed maximal degree $n$ which is 
$$
u_n\sim \frac{(3n)!}{6^{2n}(n!)^2} \exp(2 \minus \frac{2}{9n}+O(n^{\minus 2})).
$$

Lastly, as it is mentioned before, A. Degtyarev uses dessins d'enfants to compute the fundamental group when they are maximal. A dessin d'enfant is maximal if it is connected, has no monochrome vertex and there is only one singular fiber in each region. In this paper, we find a sufficient condition to deform non-maximal dessins d'enfants into maximal ones.  In Section \ref{CH:5}, we prove the following theorem:

\begin{theorem}\em
A non-maximal dessin d'enfant of a completely reducible trigonal curve can be deformed into a maximal dessin d'enfant if its regions that have more than one singular fiber include monochrome vertices that connect all the singular fibers.
\end{theorem}

\textbf{Organization of the paper.} The paper is structured as follows: In Section 2, we introduce the the main ingredients of the study, i.e. the trigonal curves, the $j$-invariant and dessins d'enfants. In Section 3, we discuss the combinatorial properties of the dessins d'enfants of CRTCs, state the combinatorial type theorem and also classify the combinatorial types up to the maximal degree $4$. Important facts on deformations of dessins are presented in Section 4. We introduce a sufficient condition where a dessin d'enfant can be maximal in Section 5 and finally we conclude our work in Section 6. 

\section{Preliminaries}\label{section:2}

In this section, we first introduce a special set of algebraic curves that we study in this paper: \textit{trigonal curves}. We then define the $j$-invariant and the dessins d'enfants of trigonal curves.

\subsection{The Trigonal Curves}
Let $\Sigma=P^1\times P^1$ and let $p:\Sigma\rightarrow P^1$ be a projection of $\Sigma$ to one of its components. Let $E$ be a section of $p$ and for each $b$ in $P^1$, let $F_b$ be the fiber over $b$.

\begin{definition} \em
A \emph{trigonal curve} is a curve $C\subset \Sigma$ not containing $E$ or a fiber of $\Sigma$ as a component such that the restriction $p: C \rightarrow P^1$ is a map of degree $3$, i.e. each fiber intersects with $C$ in at most $3$ points. In the affine part, $C$ is defined by $F(x,y)=0$ with $F(x,y)\in \mathbb{C}[x,y]$ where deg$_y F=3$. 
\end{definition}

The trigonal curves are also closely related to the plane curves with a special singularity:

\begin{remark}\em
Let $C$ be a singular plane curve and has a singular point $p$ of multiplicity of $\text{degree }C-3$. If $C$ is blown up at $p$, the proper transform of $C$ results in a trigonal curve $\tilde{C}$ on $\Sigma=P^1\times P^1$ with an exceptional section $E_p$ corresponding to $p$. 
\end{remark}

In this paper, we narrow our studies for a special subset of trigonal curves. 

\begin{definition}\em{
A trigonal curve $C$ is \textit{completely reducible} if it is defined by $F(x,y)=(y-y_1(x))(y-y_2(x))(y-y_3(x))=0$  
where $y_i \in \mathbb{C}[x]$ for all $i\in \{1,2,3\}$.}
\end{definition}

It is important to note that $F$ can have arbitrarily large degree in $x$.

\subsection{The $j$-invariant}

As it is mentioned before, a dessin can arise from finite coverings and to associate dessins with trigonal curves, here we define an orientation preserving covering $j_C: B \rightarrow P^1$ for a trigonal curve $C$.

Recall that the \emph{cross-ratio} of a quadruple of pairwise distinct points $z_1,z_2,z_3,z_4 \in \bar{\mathbb{C}} := \mathbb{C} \cup {\infty}$ is defined as
$$
  (z_1,z_2;z_3,z_4):= \frac{(z_1-z_3)(z_2-z_4)}{(z_2-z_3)(z_1-z_4)}
$$
It is invariant under \emph{M\"{o}bius transformations}
$$
  z\mapsto \frac{az+b}{cz+d},\hspace{5 mm} a,b,c,d \in \mathbb{C},\hspace{5 mm} ad-bc\neq 0.
$$
After applying some M\"{o}bius transformations, we may get
$$
   (z_1,z_2;z_3,\infty)=\frac{z_1-z_3}{z_2-z_3},\hspace{5 mm} (\lambda,1;0,\infty)=\lambda.
$$
Normally, the cross ratio depends on the order of the points. When one change the order and bring them back to the form $\lambda',1,0,\infty$, it applies the following transformations
$$\lambda \mapsto 1-\lambda, \lambda \mapsto 1/{\lambda}.$$
When we iterate these transformations, we get the following orbits:
\begin{equation}\label{orbit}
  \lambda,\hspace{5 mm} 1-\lambda, \hspace{5 mm}   \frac{1}{\lambda}, \hspace{5 mm}   \frac{1}{1-\lambda}, \hspace{5 mm} 1-{\frac{1}{\lambda}}, \hspace{5 mm}   \frac{\lambda}{\lambda-1}. 
\end{equation}


Substituting these shows that the function
\begin{equation}\label{jinv}
  j:=\frac{{4(\lambda^2-\lambda+1)^3}}{27\lambda^2(\lambda-1)^2}
\end{equation}
is invariant under (\ref{orbit}). Moreover, since (\ref{jinv}) is an equation of degree 6 in $\lambda$ for any given value of $j$, it has six solutions and they are given by (\ref{orbit}). Therefore, $j$ determines the unordered quadruple $z_1,z_2,z_3,z_4\in P^1$ uniquely up to M\"{o}bius transformation, where $\lambda$ is the cross-ratio of the unordered four points. This expression is called the \textit{$j$-invariant} of the unordered four points. Note that the $j$-invariant of a triple $z_1,z_2,z_3 \in \mathbb{C}$ means the $j$-invariant of the quadruple $z_1,z_2,z_3,\infty$.

\subsubsection{Real Values of the $j$-invariant}

For our future construction, it is important to gain more information about the real values of the $j$-invariant. For real values, we only use the  M\"{o}bius transformations that preserves $\infty \in P^1$, i.e. $z \rightarrow az+b, a\neq 0$ and in the complex plane $\mathbb{C}$, one can talk about the angles and length ratios. 
\begin{figure}[ht!]
\centering
\includegraphics[width=70mm]{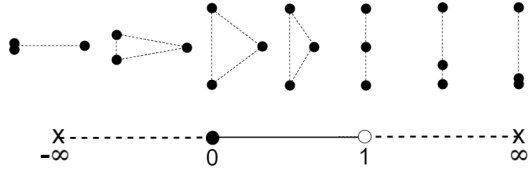}
\caption{Real values of the $j$-invariant}
\end{figure}

\begin{proposition}\em \cite{degtyarev2012topology}
 Assume that the $j$-invariant $j$ of a triple $z_1,z_2,z_3\in \mathbb{C}^1$ is real. Then one has one of the following two cases:

(1) When $j\geq 1$, the points $z_1,z_2,z_3$  are collinear. In this case, the length ratio of the two segments formed by the points (shorter to longer) decreases from $1/2$ to $0$ with $j$ increasing from $1$ to $\infty$.
  
(2) When $j\leq 1$, the points $z_1,z_2,z_3$ form an isosceles triangle. In this case, the angle $\Theta$ at the vertex decreases from $2\pi$ to $0$ with $j$ decreasing from $1$ to $-\infty$.

Conversely, the $j$-invariant of any triple satisfying (1) and (2) is real.
\end{proposition}
The following is a useful fact for the later computations.
\begin{proposition}\em
The $j$-invariant is equal to $0$ when $\lambda= e^{\pm\frac{\pi i}{3}}$, equal to $1$ when $\lambda \in \{-1,0.5,2\}$ and equal to $\infty$ when $\lambda \in \{0,1,\infty\}$.
\label{proposition::valuesofj}
\end{proposition}
\begin{proof}
$j=0$ implies $\frac{{4(\lambda^2-\lambda+1)^3}}{27\lambda^2(\lambda-1)^2}=0$. Since, $j=0$, $\lambda\neq 0,1$ so we have $4(\lambda^2-\lambda+1)^3=0$ which implies $\lambda = e^{\pm \frac{\pi i}{3}}$.

Secondly, $j=1$ implies $4(\lambda^2-\lambda+1)^3=27\lambda^2(\lambda-1)^2$, and $4(\lambda^2-\lambda+1)^3-27\lambda^2(\lambda-1)^2=0$ so we have $[(\lambda-2)(\lambda+1)(2\lambda-1)]^2=0$ which implies $\lambda=2,1/2,-1$.

Lastly, $j=\infty$ if two of the three points coincide. Hence for the triple $z_1,z_2,z_3$, either $z_1=z_2$ or $z_1=z_3$ or $z_2=z_3$. They result $\lambda=1$, $\lambda=0$ and $\lambda=\infty$ respectively.    
\end{proof}

\subsubsection{The $j$-invariant of a trigonal curve}

Let $C\subset \Sigma \rightarrow P^1$ be a trigonal curve. The \emph{$j$-invariant of a trigonal curve} is the rational map $j_C: P^1 \rightarrow P^{1}$ that sends each point $b$ in the base $P^1$ to the $j$-invariant of $(C\cup E) \cap F_b$ considered as a subset of 4 points in $F_b$.

\begin{remark}\em
For the completely reducible trigonal curves (CRTC), it is often useful to factor the $j$-invariant $j_C$ into two maps as $j_C=\phi_2 \circ \phi_1$ where $\phi_1:P^1 \rightarrow {P}^1$ takes a point $b \in P^1$ and sends it to the cross ratio $(y_1(b),y_2(b);y_3(b),\infty)$ and $\phi_2:{P}^1 \rightarrow {P}^1 $ takes the cross ratio $\lambda$ and sends it to the $\frac{4(\lambda^2-\lambda+1)^3}{27\lambda^2(\lambda-1)^2}$.
\label{remark::rem1}
\end{remark}

We deduce from Remark \ref{remark::rem1} that $\phi_2$ induces a branched covering of ${P}^1_{\mathbb{R}}$ whose branched locus is $\{0,1,\infty\}$ and the corresponding ramification indexes are $3,2,2$ respectively (See Figure \ref{crossratio}). In Figure \ref{crossratio}, pull-backs of $0,1,\infty$ are shown with $\bullet,\circ,\times$ respectively. There is also one $\times$ vertex at the point at infinity.

\begin{figure}[ht!]
\centering
\includegraphics[width=90mm]{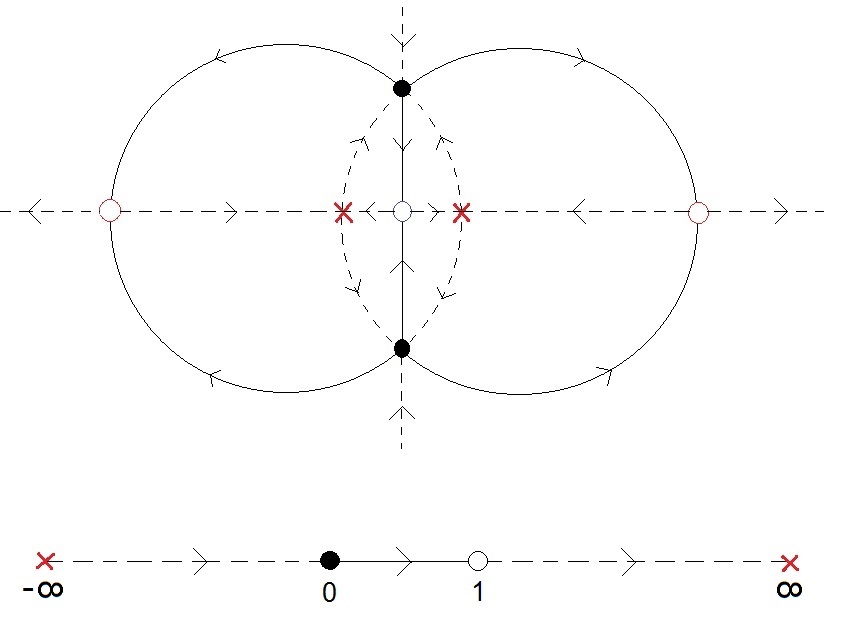}
\caption{The cross ratio graph}
\label{crossratio}
\end{figure}

\subsection{Dessins d'Enfants of Trigonal Curves}
Here, we define what a dessin d'enfant of a trigonal curve is.

\begin{definition}\em
Let $C \subset \Sigma \rightarrow P^1$ be a trigonal curve and $j_C=P^1 \rightarrow P^1$ be its $j$ invariant. The \textit{dessin d'enfant} of $C$, $\Gamma_C$, is the dessin d'enfant of the covering $j_C$, i.e. $\Gamma_C=j_C^{-1}([0,1])$.
\end{definition}

We will also mark $\times$ on the points of $j^{-1}(\infty)$ and call \textit{monochrome} to the ramification points of $j$ whose images are in $(0,1)$. 

Note that if a trigonal curve $C$ has a triple intersection, $j$ behaves constant in an $\epsilon$ neighborhood of that point for some $\epsilon > 0$. Hence, $\Gamma_C$ is not defined. Therefore, in the rest of the paper, we will consider only the trigonal curves with no triple intersections.

\section{Combinatorial Types of Dessins d'Enfants}\label{ch:comb}

In this section, we work on the combinatorial types of dessins d'enfants. We first list some combinatorial properties of a dessin $\Gamma_C$ of a CRTC with the maximal degree $n$ and also introduce an algorithm to compute $\Gamma_C$. We then define the combinatorial type of a dessin and state a theorem which gives an upper bound for the number of different combinatorial types of dessins for a fixed maximal degree. We finally give the exact combinatorial types up to the maximal degree four. 

We assume from now on $C$ is a CRTC unless otherwise stated. 

\subsection{Dessins d'Enfants of CRTCs} \label{dessin-crtc}

The main object in this section is the dessins d'enfants of the completely reducible trigonal curves. Our results for these dessins mostly depend on the \textit{maximal degrees} of CRTCs. Hence we first define the maximal degree before stating the results.

\begin{definition}\em
Let $C$ be in the form $(y-y_1)(y-y_2)(y-y_3)=0$ where $y_1,y_2,y_3\in \mathbb{C}[x]$. Let $d_i$ be degree of $y_i$ for $i\in \{1,2,3\}$ and define 
$$
\mathfrak{d}(C):= \textrm{maximum}\{d_1,d_2,d_3\}.
$$
$\mathfrak{d}(C)$ is called the \textit{maximal degree} of $C$.
\end{definition}

\subsubsection{Combinatorial Properties of the Dessins d'Enfants of CRTCs}

In this section, we give some combinatorial properties of the dessins d'enfants of CRTCs. The first property is about the number of vertices and edges of $\Gamma_C$.

\begin{proposition}
{\em
$\Gamma_C$ has at most $2n$ $\bullet$- vertices, at most $3n$ $\circ$- vertices and exactly $6n$ edges.}
\label{Prop6n}
\end{proposition}

\begin{proof}
From Remark \ref{remark::rem1}, we have $j_C=\phi_2 \circ \phi_1$. When we look at the $\phi_2^{-1}[0,1]$, we get the cross-ratio graph as in Figure \ref{crossratio}. From Proposition \ref{proposition::valuesofj}, the cross-ratio graph has 2 $\bullet$- vertices and 3 $\circ$- vertices. The dessin d'enfant is the pull-back of the cross-ratio graph under $\phi_1^{-1}$. There are two cases for $\phi_1$:

\textbf{Case 1}: Assume that $\phi_1$ is an $n$-fold unbranched covering. Then the equations in Proposition \ref{proposition::valuesofj} have no repeated roots. This implies that $\Gamma_C$ has exactly $2n$ $\bullet$- vertices and $3n$ $\circ$- vertices. Since $\phi_1$ is unbranched and the cross-ratio graph has 6 edges, $\Gamma_C$ has $6n$ edges.

\textbf{Case 2}: Assume $\phi_1$ is a branched covering. This can happen when the equations in Proposition \ref{proposition::valuesofj} have repeated roots, then the $\Gamma_C$ has $< 2n$ $\bullet$- vertices and $ < 3n$ $\circ$- vertices. Moreover, if $\Gamma_C$ has a monochrome vertex, we count the corresponding edge as one edge, not two edges. Hence, we can assume that $\phi_1$ is branched only on the vertices of the cross-ratio graph and this implies that $\Gamma_C$ has $6n$ edges in total. 

\end{proof}

Moreover, from Proposition \ref{proposition::valuesofj}, we can easily deduce that there are three different types of $\circ$- vertices; one type is coming from $\lambda=-1$, one is coming from $\lambda=0.5$, and one is coming from $\lambda=2$. We use color codes for each type; say $red$, $blue$ and $green$ $\circ$- vertices respectively. Similarly, there are two different types of $\bullet$- vertices; one is coming from $j=e^{\frac{i\pi}{3}}$ and other is coming from $j=e^{-\frac{i\pi}{3}}$. Color them as $cyan$ and $yellow$ respectively. See Figure \ref{fig::3} as an example.

\begin{figure}[ht!]
\centering
\includegraphics[width=50mm]{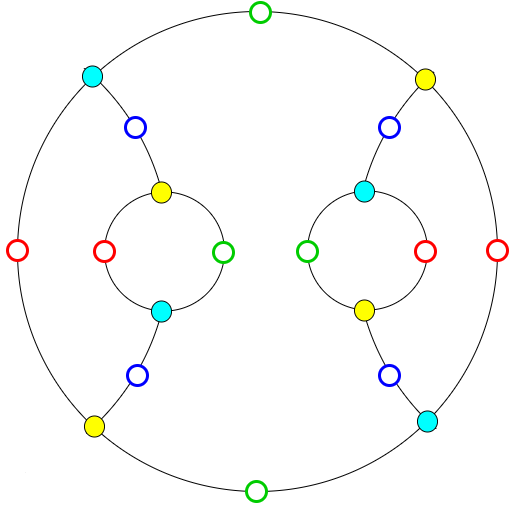}
\caption{Colored dessin d'enfant of the curve $(y-x^4)(y-x^2+1)(y-8x^2+16)=0$ }
\label{fig::3}
\end{figure}

Since each vertex in a dessin d'enfant is colored, one can also color a region using the colors of its $\circ$- vertices. 
 
\begin{proposition}\em
Each region in $\Gamma_C$ corresponds to two colors associated with the colors of the $\circ$- vertices on its boundary edges: Red-Blue (RB), Blue-Green (BG), Red-Green (RG). Furthermore, there are $2n$ edges for each region type.
\label{RBG}
\end{proposition}

\begin{proof}
Let $R$ be a region in $\Gamma_C$. We have $j_C=\phi_2 \circ \phi_1$ from Remark \ref{remark::rem1}. Since the interior of $R$, $R^{\circ}$, is connected, its image under $\phi_1$, $\phi_1(R^{\circ})$, should also be connected. Hence, $\phi_1(R^{\circ})$ is either an $RB$, $BG$ or $RG$ region in the cross-ratio graph. WLOG, let the image be an $RB$ region. Since the region has only red and blue $\circ$- vertices on its boundary, its preimage should also have these colors of $\circ$- vertices on its boundary. Hence, $R$ is an $RB$ region. Therefore, there are exactly two different colors of $\circ$- vertices in a region of $\Gamma_C$.

Furthermore, since $\phi_1$ is an $n$-fold covering and each region in the cross-ratio graph is bigonal, then there are $2n$ edges for each region type.
\begin{figure}[ht!]
\centering
\includegraphics[width=40mm]{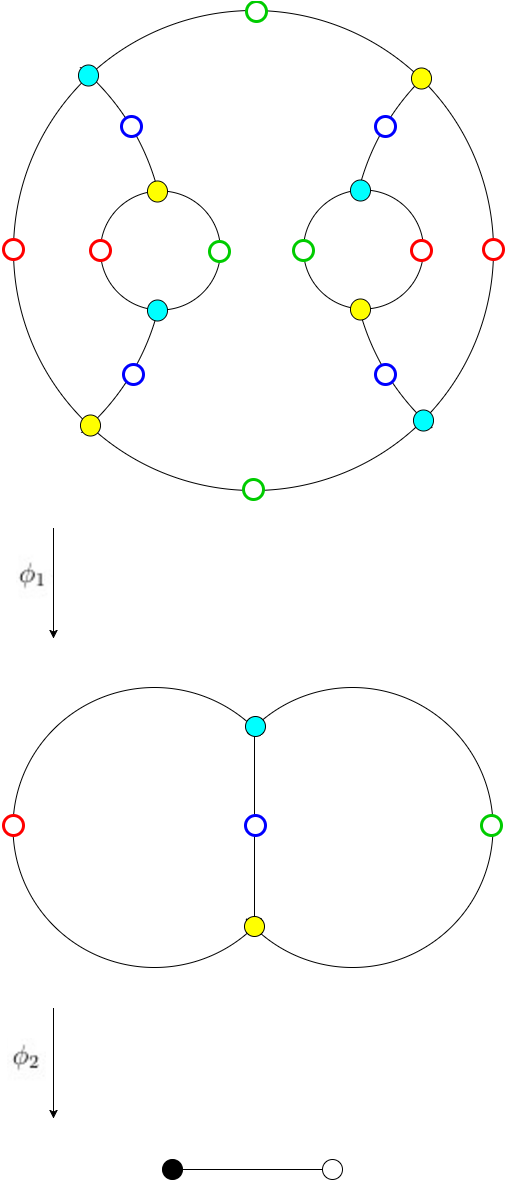}
\caption{Covering maps $\phi_1$ and $\phi_2$ that construct the $j$-invariant}
\label{fig:jcons}
\end{figure}
\end{proof}

Finally, we present upper and lower bounds for the number of regions in $\Gamma_C$.

\begin{proposition}\em 
Let $R_{\Gamma_C}$ be the number of regions of $\Gamma_C$. Then, 
\begin{center}
$n+2 \leq R_{\Gamma_C}\leq 3n$.
\end{center}
\label{numberofregions}
\end{proposition}

\begin{proof}
From Proposition \ref{RBG}, $\Gamma_C$ has totally $6n$ edges where each region type has $2n$ edges. For a fixed region type, the maximum number of regions happens is obtained when there are $n$ bigonal regions, hence $R_{\Gamma_C}\leq 3n$ since there are three region types at all. 

Moreover, let $V_{\Gamma_C},E_{\Gamma_C}$ and $M_{\Gamma_C}$ be the number of vertices, edges and connected components of $\Gamma_C$ respectively. To get a complete triangulation, we need to connect each connected components with adding extra edges between them without changing $R_{\Gamma_C}$. Hence, total number of edges we have is $E_{\Gamma_C}+M_{\Gamma_C}-1$ and the topological Euler characteristic of $\Gamma_C$ is 
$$\chi(\Gamma_C)=V_{\Gamma_C}-(E_{\Gamma_C}+M_{\Gamma_C}-1)+R_{\Gamma_C}=2$$
and this implies 
$$R_{\Gamma_C}=M_{\Gamma_C}+E_{\Gamma_C}-V_{\Gamma_C}+1.$$
From Proposition \ref{Prop6n}, $V_{\Gamma_C}\leq 5n$ and $E_{\Gamma_C}=6n$. Thus we have 
$$R_{\Gamma_C} \geq M_{\Gamma_C}+n+1.$$ 
Moreover, we also have $M_{\Gamma_C}\geq 1$, so we get 
$$
R_{\Gamma_C}\geq n+2
$$
\end{proof}

\subsubsection{Computing Dessins d'Enfants}
We use computer experimentation to obtain insights about the dessins d'enfants and how they behave under the deformation of curves. We improve an algorithm which takes the components of a CRTC as input and computes its dessin d'enfant as the output. Here is the algorithm:

For a given CRTC in the form $(y-y_1)(y-y_2)(y-y_3)=0$ where $y_1,y_2,y_3\in \mathbb{C}[x]$, the very first step is to find the position of the $\times$-vertices i.e. singular fibers of the complement. These fibers are passing through either singular points of one irreducible component or intersections of these components. In our case, any irreducible component has no singular point by itself since ${\partial (y-y_i)}/{\partial y} \neq 0$. Hence we only need to find the intersection points of the three components $\{y-y_i\}, i \in \{1,2,3\}$. Hence, we set $y_i=y_j$ for $i\neq j$ and $i,j\in\{1,2,3\}$. The solution set of these three equalities gives the singular fibers i.e. position of $\times$-vertices.

Since a dessin is the preimage of the $j$-invariant of $C$, $j_C$, over $P_{[0,1]}^1$, it is enough to find points in $P^1$ where $j_C=r$ with $r\in [0,1]$. To do this, we first compute the cross-ratio in the rational function form $\lambda=\frac{y_1-y_3}{y_2-y_3}$. Then, we set $j_C=\frac{{4(\lambda^2-\lambda+1)^3}}{27\lambda^2(\lambda-1)^2}=r$, substitute $\lambda$ in the equation, solve it for $x$ and plot the solution points. For $r=0$ or $r=1$, we plot $\bullet$- or $\circ$- vertices respectively.

Finally, to compute the dessin, we divide the interval $[0,1]$ into equal length sub-intervals $0=r_0<r_1<...<r_n=1$. For example, we may take $n=100$ or even a bigger number can be chosen if needed. Then for each point $r_i$ in this partition, we apply the previous process and plot the points and compute the dessin d'enfant. An example can be found in Figure \ref{fig:sample}. 

\begin{figure}[ht!]
\centering
\includegraphics[width=80mm]{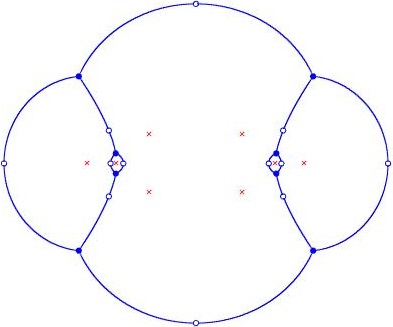}
\caption{Dessin d'enfant of the curve $(y-x^4)(y-x^2+1)(y-8x^2+16)=0$. Red $\times$-vertices are the points where singular fibers exist.}
\label{fig:sample}
\end{figure}

Since dessins d'enfants are embedded in $P^1=\mathbb{C}\cup \infty$, we just forget the point at infinity and plot the rest on the complex plane. However, it is easy to understand what happens on the point at infinity by just looking the behavior of the rest of the dessin.

Note that one can also plot color dessins d'enfants using our algorithm and Proposition \ref{proposition::valuesofj}.

\subsection{The Combinatorial Type Theorem}\label{combtype}

In this section, we state the combinatorial type theorem. We first define the combinatorial type of a dessin:

\begin{definition}\em
Each region of the complement of a dessin $\Gamma_C$ is an $m$-gon, where $m$ is equal to the number of the boundary edges. The {\it combinatorial type} of $\Gamma_C$ is the list of these $m$'s written in non-increasing order.
\end{definition}

For example, the dessin d'enfant in the Figure \ref{Fig5} has totally five regions, two of them is 6-gonal and three of them is bigonal. Hence, its combinatorial type is $[6,6,2,2,2]$. We omit the bivalent $\circ$-vertices while drawing dessins d'enfants for simplicity.

\begin{example}\em
Let $C$ be a completely reducible trigonal curve defined by $(y-x^3)(y+x^2)(y-1)=0$. Then $\Gamma_C$ is as in Figure \ref{Fig5}.

\begin{figure}[h!]
\centering
\includegraphics[width=40mm]{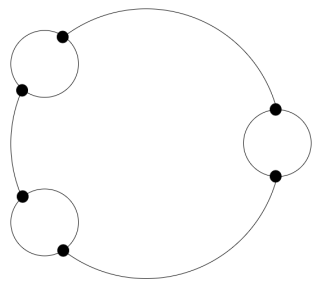}
\caption{Dessin of $(y-x^3)(y+x^2)(y-1)=0$.}
\label{Fig5}
\end{figure}

\end{example}

In the following theorem, we find an upper bound for the number of different combinatorial types of dessins d'enfants for a fixed maximal degree $n$. 

\begin{theorem}
\em Let $Comb(n)$ be the set of all combinatorial types of dessins of $C$. Then, we have 
\begin{equation}
|Comb(n)|\leq \frac{p(n)(p(n)+1)(p(n)+2)}{6} - K(n)
\label{mainthm}
\end{equation}with
$$
K(n)=\sum_{r=1}^{\lfloor \frac{n+1}{3} \rfloor}p(n,r)\sum_{j=r}^{\lfloor \frac{n+1-r}{2} \rfloor,}p(n,j)\sum_{i=j}^{n-j-r+1} p(n,i).
$$
where $p(n)$ is the number of integer partitions of $n$, $p(n,m)$ is the number of integer partitions of $n$ with length $m$.
\end{theorem} 

\begin{proof}
Let $\Gamma\in \Gamma_\text{CRTC}(n)$. From Proposition \ref{RBG}, $\Gamma$ has three region types $RB,BG$ and $RG$ and $2n$ edges for each type. Colors of regions have no importance for the decomposition. Hence, WLOG, we work on $RG$ regions. Since $\Gamma$ has $2n$ edges in its $RG$ regions, it may have, for example, one $2n$-gonal $RG$ region only, or one $2(n-1)$-gonal and one bigonal $RG$ regions. More generally, for any integer partition of $n=n_1+n_2+...+n_r$, it is possible to have one $2n_1$-gonal $RG$ region, one $2n_2$-gonal $RG$ region, and so on till one $2n_r$-gonal $RG$ region. Hence, there are totally $p(n)$ different possible cases for $RG$ regions where $p(n)$ is the number of possible integer partitions of $n$. Similarly, this is true for the other region types. 

We order the integer partitions of a positive integer $n$ as $P(n)=\{\{n\},\{n-1,1\},\{n-2,2\},\{n-2,1,1\},...,\{1,...,1\}\}$.

Recall that the number of regions and the number of edges for each region define a combinatorial type. Hence, for example, the rows in Table \ref{table:class}-(A) show all possible combinatorial types, we call them \textit{pre-combinatorial types}, that have one $2n$-gonal $RB$ region, one $2n$-gonal $BG$ region. The different types in $RG$ regions can be deduced from the different integer partitions of $n$. Hence, there are totally $p(n)$ different pre-combinatorial types that have one $2n$-gonal $RB$ region and one $2n$-gonal $BG$ region. 

\begin{table}[H]
\centering
\caption{Different combinations to define the combinatorial types}
\begin{subtable}{.3\textwidth}
\centering
 \begin{tabular}{|c|c|c|}
 \hline
 $RB$ & $BG$ & $RG$\\
 \hline
 2n & 2n & 2n\\
 2n & 2n & 2(n-1),2\\
 \vdots & \vdots & \vdots \\ 
 2n & 2n & 2,...,2 \\
 \hline
\end{tabular}
\caption{}
\end{subtable}%
\begin{subtable}{.33\textwidth}
\centering
\begin{tabular}{|c|c|c|}
 \hline
 $RB$ & $BG$ & $RG$\\
 \hline
 2n & 2(n-1),2 & 2(n-1),2\\
 2n & 2(n-1),2 & 2(n-2),4\\
 \vdots & \vdots & \vdots \\ 
 2n & 2(n-1),2 & 2,...,2 \\
 \hline
\end{tabular}
\caption{}
\end{subtable}%
\begin{subtable}{.33\textwidth}
\centering
\begin{tabular}{|c|c|c|}
 \hline
 $RB$ & $BG$ & $RG$\\
 \hline
 p(n,1) & p(n,1) & p(n,1)\\
 p(n,2) & p(n,2) & p(n,2)\\
 \vdots & \vdots & \vdots \\ 
 p(n,n) & p(n,n) & p(n,n) \\
 \hline
\end{tabular}
\caption{}
\end{subtable}

\label{table:class}
\end{table}

Similarly, we can list the pre-combinatorial types that have one $2n$-gonal $RB$ region, one $2(n-1)$ and one $2$-gonal $BG$ regions as in Table \ref{table:class}-(B). There are $p(n)-1$ pre-combinatorial types (We do not take the case of one $2n$-gonal $RG$ region into consideration to avoid a repetition since it is already counted in the previous case). If we continue to construct the pre-combinatorial types that has one $2n$-gonal $RB$ region, we get totally $\sum_{i=1}^{p(n)} i$ types. 

More generally, let $n_k=\{n_{k,1},...,n_{k,r}\}$ be the $k$-th partition in $P(n)$. The number of pre-combinatorial types that have a $2 n_{k,1}$-gonal $RB$ region, ..., a $2 n_{k,r}$-gonal $RB$ region is $\sum_{i=1}^{p(n)-k+1} i$. Therefore, there are totally
\begin{equation}
\sum_{i=1}^{p(n)} i + \sum_{i=1}^{p(n)-1} i + ... + \sum_{i=1}^{1} i = \frac{p(n)(p(n)+1)(p(n)+2)}{6}
\label{equ1}
\end{equation}
pre-combinatorial types.
 
From Proposition \ref{numberofregions}, we have $R_{\Gamma_C}\geq n+2$. Hence, we need to eliminate the pre-combinatorial types that have less than $n+2$ regions. 

We first group all possible cases in each $RB,BG$ and $RG$ types in terms of the number of the regions. For each region type, we get $n$ different groups that have $p(n,1), p(n,2), ..., p(n,n-1), p(n,n)$ different pre-combinatorial types where $p(n,m)$ is the number of integer partitions of $n$ with length $m$, see Table \ref{table:class}-(C).

We first find the number of the pre-combinatorial types that have 1 $RB$ region and 1 $BG$ region and have less than $n+2$ regions in total which is
$$
p(n,1)p(n,1)\sum_{i=1}^{n-1} p(n,i).
$$
Similarly, the number of pre-combinatorial types that have 1 $RB$ region and $m_1$ $BG$ regions where $m_1<n$ and have less than $n+2$ regions in total is
$$
p(n,1)p(n,m_1)\sum_{i=m_1}^{n-m_1} p(n,i).
$$
However, for $m_1> \lfloor \frac{n}{2} \rfloor$, the permutation will start to repeat. Hence, the number of pre-combinatorial types that have 1 $RB$ region and have totally less than $n+2$ regions is
$$
p(n,1)\sum_{j=1}^{\lfloor \frac{n}{2} \rfloor,}p(n,j)\sum_{i=j}^{n-j} p(n,i).
$$
Similarly, the number of the pre-combinatorial types, that have 2 $RB$ regions and less than $n+2$ regions is
$$
p(n,2)\sum_{j=2}^{\lfloor \frac{n-1}{2} \rfloor,}p(n,j)\sum_{i=j}^{n-j-1} p(n,i).
$$
More generally, the number of the pre-combinatorial types that have  $r$ $RB$ regions and less than $n+2$ regions is
$$
p(n,r)\sum_{j=r}^{\lfloor \frac{n+1-r}{2} \rfloor,}p(n,j)\sum_{i=j}^{n-j-r+1} p(n,i).
$$
Note that when there exists more than $r=\lfloor \frac{n+1}{3} \rfloor$ $RB$ regions, the same pre-combinatorial types will be obtained. Therefore, the total number of the pre-combinatorial types that have less than $n+2$ regions is

\begin{equation}
\sum_{r=1}^{\lfloor \frac{n+1}{3} \rfloor}p(n,r)\sum_{j=r}^{\lfloor \frac{n+1-r}{2} \rfloor,}p(n,j)\sum_{i=j}^{n-j-r+1} p(n,i).
\label{equ2}
\end{equation}

If we combine (\ref{equ1}) and (\ref{equ2}), we obtain the main result.
\end{proof}

By the previous theorem, we find not only an upper bound for the number of the different combinatorial types for a fixed maximal degree, but also all pre-combinatorial types that construct this upper bound in the proof. However these pre-combinatorial types may not be an actual dessin d'enfant. In the next section, we list exact types up to the maximal degree 4.
 
\subsection{Combinatorial Types for Some Special Maximal Degrees}\label{spec_deg}

In this section, we list all possible combinatorial types for $\mathfrak{d}(C) \in \{1,2,3,4\}$. For each maximal degree, we first use the combinatorial type theorem to get all possible types and then prove whether they are realizable based on our experiment results. 

\begin{proposition}\em
If $\mathfrak{d}=1$, the only combinatorial type is $[2,2,2]$.
\end{proposition}

\begin{proof}
From the combinatorial type theorem, we get $|Comb(1)|\leq 1$ and the type is $[2,2,2]$. We realize this type as the dessin d'enfant of the curve $C:(y-x)(y+x)(y-1)=0$.
\end{proof}

\begin{figure}[h!]
\centering
\includegraphics[width=13mm]{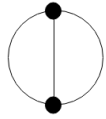}
\caption{An example for the combinatorial type $[2,2,2]$ for $\mathfrak{d}(C)=1$.}
\label{d1dessin}
\end{figure}

\begin{proposition}\em
If $\mathfrak{d}=2$, there are three combinatorial types as $[4,4,2,2],[4,2,2,2,2]$ and $[2,2,2,2,2,2]$.
\end{proposition}

\begin{proof}
From the combinatorial type theorem, we get $|Comb(2)|\leq 3$ and the pre-combinatorial types as $[4,4,2,2],[4,2,2,2,2]$ and $[2,2,2,2,2,2]$. Trigonal curve and the corresponding dessin examples for each type can be found in Table \ref{table:1} and Table \ref{fig:decomp2} respectively.
\begin{table}[H]
\centering
\caption{Combinatorial types and example for each type when $\mathfrak{d}(C)=2$}
\begin{tabular}{||P{4cm} || P{6cm}||} 
 \hline
 Decomposition type & Examples \\ [0.5ex] 
 \hline\hline
$[4,4,2,2]$ & $(y-x^2+1)(y+x)(y-x)$ \\ 
  $[4,2,2,2,2]$ & $(y-x^2+1)(y+x)(y-x-4)$ \\
 $[2,2,2,2,2,2]$ & $(y-x^2+1)(y+x+0.25)(y-x+0.25)$   \\[1ex] 
 \hline
\end{tabular}

\label{table:1}
\end{table}
 
\end{proof}
\begin{table}[H]
\begin{center}
\caption{Dessin examples for each combinatorial type when $\mathfrak{d}(C)=2$.}
\begin{tabular}{|m{4cm}| m{4cm}| m{4cm}|} 
\hline
\centering\arraybackslash\includegraphics[width=30mm]{DessinDelta2_2.png} & \centering\arraybackslash\includegraphics[width=25mm]{DessinDelta2_1.png} & \centering\arraybackslash\includegraphics[width=24mm]{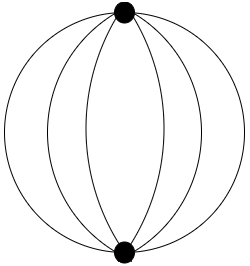} \\ \hline
\centering\arraybackslash$[4,4,2,2]$ & \centering\arraybackslash$[4,2,2,2,2]$ & \centering\arraybackslash$[2,2,2,2,2,2]$ \\ \hline
\end{tabular}

\label{fig:decomp2}
\end{center}
\end{table}

\begin{proposition}\em
If $\mathfrak{d}=3$, there are 8 different combinatorial types.
\end{proposition}

\begin{proof}
From the combinatorial type theorem, we get $|Comb(3)|\leq 8$. The types and example for each type can be found in Table \ref{table:dec}.
\begin{table}[H]
\centering
\caption{Combinatorial types and example for each type when $\mathfrak{d}(C)=3$}
\begin{tabular}{||P{4cm} || P{10cm}||} 
 \hline
 Decomposition Type & Examples \\ [0.5ex] 
 \hline\hline
 $[6,6,2,2,2]$& $(y-x^3-x^2-1)(y+2x^2+2)(y+2)$\\
 $[6,4,4,2,2]$& $(y-x^3-x^2-1)(y+2x^2-1)(y+2)$\\
 $[6,4,2,2,2,2]$& $(y-x^3-x^2-1)(y+2x^2)(y+2)$ \\
 $[6,2,2,2,2,2,2]$& $(y-x^3-x^2)(y-x^3-2x^2+x)(y+1.5)$\\
 $[4,4,4,2,2,2]$& $(y-2x^3+3x^2+6x-2)(y+4x^2+2x-3)(y-3x^2+x-3.5)$\\
 $[4,4,2,2,2,2,2]$ & $(y-x^3-3x^2-x-1)(y+x^2-x-2)(y+2x^3+9x^2+3x-2)$\\
 $[4,2,2,2,2,2,2,2]$ & $(y-x^3-3x^2-x-1)(y+x^2-x-2)(y+2x^3+9x^2+3.06x-2)$\\
 $[2,2,2,2,2,2,2,2,2]$ & $(y-x^3+3x-1)(y-3x^2+3x)(y)$  \\[1ex] 
 \hline
\end{tabular}

\label{table:dec}
\end{table}

\end{proof}

\begin{center}

\begin{table}[H]
\centering
\caption{Dessin examples for each combinatorial type when $\mathfrak{d}=3$}
\begin{tabular}{|m{4cm}| m{4cm}| m{4cm}|} 

\hline 
\centering\arraybackslash\includegraphics[width=33mm]{DessinDelta3_4.png} & \centering\arraybackslash\includegraphics[width=34mm]{DessinDelta3_2.png} & \centering\arraybackslash\includegraphics[width=35mm]{DessinDelta3_3.png} \\ \hline
\centering\arraybackslash$[6,6,2,2,2]$ & \centering\arraybackslash$[6,4,4,2,2]$ & \centering\arraybackslash$[6,4,2,2,2,2]$ \\ \hline
\centering\arraybackslash\includegraphics[width=35mm]{DessinDelta3_1.png} & \centering\arraybackslash\includegraphics[width=34mm]{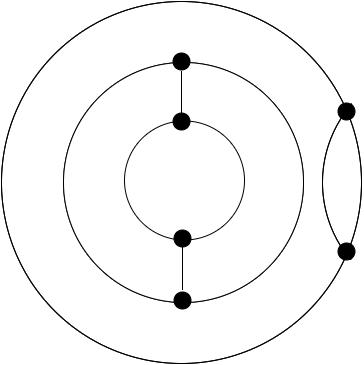} & \centering\arraybackslash\includegraphics[width=34mm]{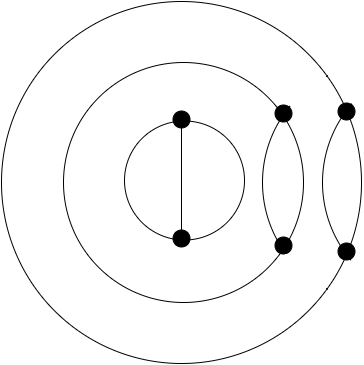} \\ \hline
\centering\arraybackslash$[6,2,2,2,2,2,2]$ & \centering\arraybackslash$[4,4,4,2,2,2]$ & \centering\arraybackslash$[4,4,2,2,2,2,2]$ \\ \hline
\centering\arraybackslash\includegraphics[width=34mm]{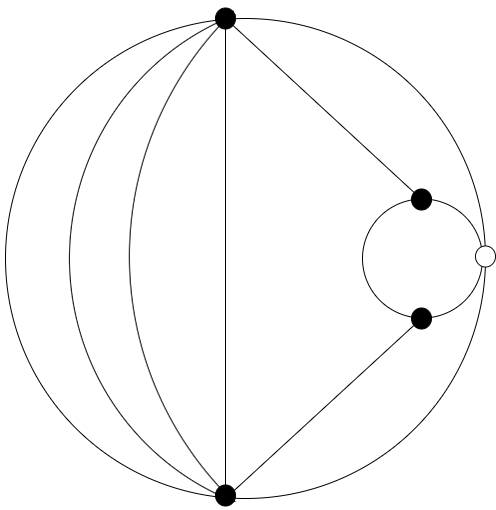} & \centering\arraybackslash\includegraphics[width=35mm]{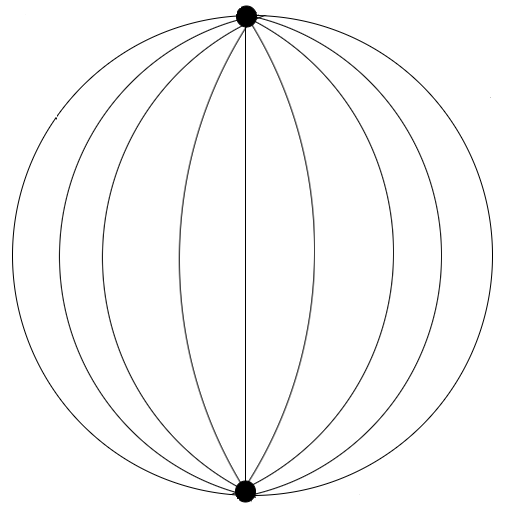} &  \\ \hline
\centering\arraybackslash$[4,2,2,2,2,2,2,2]$ & \centering\arraybackslash$[2,2,2,2,2,2,2,2,2]$ &  \\ \hline
\end{tabular}

\end{table}
\end{center}

For $\mathfrak{d}=4$, there are $23$ pre-combinatorial types from Theorem \ref{mainthm} but not all of them represent a combinatorial type.
\begin{proposition}\label{prop:19}\em
$[6,4,4,4,4,2]$ is not a combinatorial type.
\end{proposition}

\begin{proof}
There are either 6, if two $\bullet$-vertices are merged, or 8 $\bullet$- vertices. 
\begin{table}[H]
\centering
\caption{Different cases in the proof of Proposition \ref{prop:19}}
\begin{tabular}{|m{4cm}| m{4cm}| m{4cm}|} 

\hline 
\centering\arraybackslash\includegraphics[width=34mm]{case1_1_a.png} & \centering\arraybackslash\includegraphics[width=35mm]{yhexa1n.png} & \centering\arraybackslash\includegraphics[width=35mm]{case2_1_an.png} \\ \hline
\centering\arraybackslash a. Case 1 & \centering\arraybackslash b. Case 2.1 & \centering\arraybackslash c. Case 2.2.1 \\ \hline
\centering\arraybackslash\includegraphics[width=35mm]{case2_1_abn.png} & \centering\arraybackslash\includegraphics[width=34mm]{case2_2_an.png} & \centering\arraybackslash\includegraphics[width=40mm]{case2_2_bn.png} \\ \hline
\centering\arraybackslash d. Case 2.2.1 & \centering\arraybackslash e. Case 2.2.2.1 & \centering\arraybackslash f. Case 2.2.2.1 \\ \hline
\end{tabular}

\label{table:predessin}
\end{table}

\textbf{Case 1:} Let assume that there are 6 $\bullet$-vertices. The 6-gonal region is located as in Table \ref{table:predessin}-a. WLOG, we assume that $v_2$ and $v_5$ are merged vertices and have two edges between them (see the dotted edges in Figure \ref{table:predessin}). Now, there exist two four-gonal regions. However there are 4 four-gonal regions in the pre-combinatorial type $[6,4,4,4,4,2]$ and there is no other way to put another four-gonal region on Figure \ref{table:predessin}-a. Hence, we get a contradiction.

\textbf{Case 2:} Let assume that there are 8 $\bullet$-vertices. We first locate the hexagon as in Figure \ref{table:predessin}-b. There are two different cases to locate a 4-gonal region.

\textbf{Case 2.1:} Let assume the 4-gonal region shares 2 edges with the hexagon. WLOG, we assume that there are the edges $v_7-v_2$ and $v_7-v_4$ (see Figure \ref{table:predessin}-b). Then, $v_3$ cannot have an edge since we cannot put an edge inside the hexagon. This is a contradiction since each $\bullet$- vertex is trivalent.

\textbf{Case 2.2:} Let assume the 4-gonal region shares only one edge with the hexagon. In this case, there are two sub-cases as well.

\textbf{Case 2.2.1:} There are two edges between $v_7$ and $v_8$. WLOG, we assume there are edges $v_7-v_2$ and $v_8-v_3$ (see Figure \ref{table:predessin}-d). Then, we cannot add more edges to $v_2,v_3,v_7$ and $v_8$ since they are trivalent. The only way to have another four-gonal region is adding an edge between $v_1$ and $v_4$ as in Figure \ref{table:predessin}-d. However we cannot put any other four-gonal region in to this graph. Hence, we get a contradiction.

\textbf{Case 2.2.2:} There is only one edge between $v_7$ and $v_8$. Since we assume that the first square is inserted into the graph shares one edge with the hexagon, WLOG, we can have edges the $v_7-v_2$ and $v_8-v_3$. 

Here, we have two sub-cases:

\textbf{Case 2.2.2.1:} We have the edge $v_7-v_6$ (see Figure \ref{table:predessin}-e). In this case, we cannot add any edge to the vertex $v_1$, since it has to be trivalent. This is a contradiction.

\textbf{Case 2.2.2.2:} Assume there is the edge $v_7-v_4$ (see Figure \ref{table:predessin}-f). Similarly, we cannot add any edge to the vertex $v_8$. This is also a contradiction. 

Therefore, the pre-combinatorial type $[6,4,4,4,4,2]$ is not a combinatorial type.

\end{proof}

\begin{proposition}\em
When $\mathfrak{d}=4$, there are 22 different combinatorial types.
\end{proposition}
\begin{proof}
From the combinatorial type theorem, we get $|Comb(4)|\leq 23$. From Proposition \ref{prop:19}, $[6,4,4,4,4,2]$ is not a combinatorial type. The remaining pre-combinatorial types and example for each type can be found in Table \ref{table:11} and Table \ref{table:2}.
\end{proof}

\begin{table}[h!]
\centering
\caption{Combinatorial types and example for each type when $\mathfrak{d}(C)=4$}
\begin{tabular}{||P{5cm} || P{7cm}||} 
 \hline
 Decomposition types & Examples \\ [0.5ex] 
 \hline\hline
 $[8,8,2,2,2,2]$& $(y-x^4-2x^3-x^2-9x-126)(y-3x^4+3x^3+4x^2-4.4x+0.5)(y+2x^4-3x^3-x^2-2x-4)$\\
 \hline
 $[8,6,4,2,2,2]$& $(y-x^4-3x^3+3x^2-3x+3)(y-2x^4+2x^3-2x^2-2x+1)(y+x^4-x^3+x^2-x+1)$\\
 \hline
 $[8,6,2,2,2,2,2]$& $(y-x^4-3x^3+3x^2-3x+3)(y-2x^4+2x^3-2x^2+2x+1)(y+x^4-x^3+x^2-x+1)$\\
 \hline
 $[8,4,4,4,2,2]$& $(y-x^4)(y-2x^2+1)(y-8x^2+16)$\\
 \hline
 $[8,4,4,2,2,2,2]$& $(y-x^4+0.8x^3+6x^2-13)(y+x^3-x^2-x)(y-8x^2-9x+5)$\\
 \hline
 $[8,4,2,2,2,2,2,2]$& $(y-x^4-3x^3+3x^2-3x+3)(y-2x^4+2x^3-2x^2+0.5x+1)(y+x^4-x^3+x^2-x+1)$\\
 \hline
 $[8,2,2,2,2,2,2,2,2]$ & $(y+15x^4-3x^3-3x^2-3x+3)(y-2x^4-2x^2-2x+1)(y+x^4-x^3-x^2-x+1)$\\
 \hline
 $[6,6,6,2,2,2]$ & $(y-x^4+0.8x^3+6x^2+10)(y+x^3-x^2-x-6)(y-8x^2-9x+16)$\\
 \hline
 $[6,6,4,4,2,2]$ & $(y-x^4+3x+2)(y+0.5x^4-1.5x^3-x^2-6)(y+x^3-8x^2+16)$\\
 \hline
 $[6,6,4,2,2,2,2]$ & $(y-x^4-3x^3+3x^2-3x+3)(y-2x^4-8.5x^3-2x^2-2x+1)(y+x^4-x^3+x^2-x+1)$\\
 \hline
 $[6,6,2,2,2,2,2,2]$ &  $(y-x^4-3x^3+3x^2-3x+3)(y-2x^4-10.175x^3-2x^2-2x+1)(y+x^4-x^3+x^2-x+1)$\\
 \hline
 $[6,4,4,4,2,2,2]$ & $(y-x^4+3x+2)(y+0.5x^4-1.5x^3-x^2-6)(y-8x^2+16)$\\
 \hline
 $[6,4,4,2,2,2,2,2]$ & $(y-x^4-3x^3+3x^2-3x+3)(y-2x^4+2x^3-2x^2+3x+0.95)(y+x^4-x^3+x^2-x+1)$\\
 \hline
 $[6,4,2,2,2,2,2,2,2]$ & $(y-x^4-6.63x)(y-2x^2-0.35x+11)(y-8x^2+16)$\\
 \hline
 $[6,2,2,2,2,2,2,2,2,2]$ & $(y-x^4-2x^3-x^2-9x-4.9)(y-3x^4+3x^3+4x^2-4.4x+0.5)(y+2x^4-3x^3-x^2-2x-4)$\\
 \hline
 $[4,4,4,4,4,4]$ & $(y-x^4-2x^3+7x^2+1x-9)(y-2x^4+x^3-3x^2-4x+7)(y+x^4-2x^3-x^2-3x-1)$\\
 \hline
 $[4,4,4,4,4,2,2]$ & $(y-x^4-2x^3+13x^2-2x-9)(y-2x^4+x^3-3x^2-3x+7)(y+x^4-2x^3-x^2-3x-1)$\\
 \hline
 $[4,4,4,4,2,2,2,2]$ & $(y-x^4)(y-2x^2+5)(y-8x^2+16)$\\
 \hline
 $[4,4,4,2,2,2,2,2,2]$ & $(y-x^4+5.5)(y-2x^2+5)(y-8x^2+16)$\\
 \hline
 $[4,4,2,2,2,2,2,2,2,2]$ & $(y-x^4+4)(y-2x^2+1)(y-8x^2+16)$\\
 \hline
 $[4,2,2,2,2,2,2,2,2,2,2]$ & $(y-x^4+8.5)(y-2x^2+1)(y-8x^2+16)$\\
 \hline
 $[2,2,2,2,2,2,2,2,2,2,2,2]$ & $(y-x^4+6x^2-4x)(y-4x^3+6x^2-1)(y)$\\
 \hline
\end{tabular}

\label{table:11}
\end{table}

\begin{table}
\centering
\caption{Dessin examples for each combinatorial type when $\mathfrak{d}=4$}
\begin{tabular}{|m{3.25cm} |m{3.25cm}| m{3.25cm}|m{3.25cm}|} 
\hline 
 
\centering\arraybackslash\includegraphics[width=26mm]{DessinDelta4_1.png} & \centering\arraybackslash\includegraphics[width=27mm]{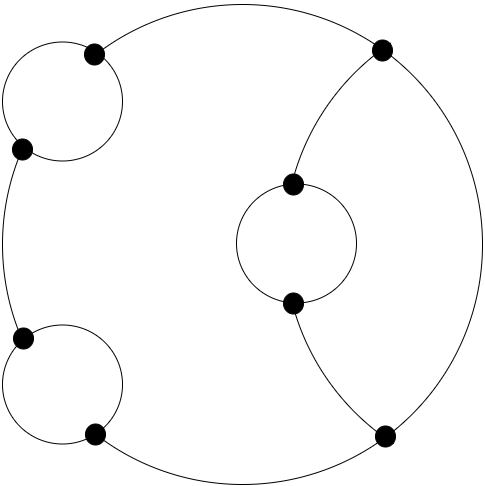} & \centering\arraybackslash\includegraphics[width=27mm]{DessinDelta4_3.png} & 
\centering\arraybackslash\includegraphics[width=27mm]{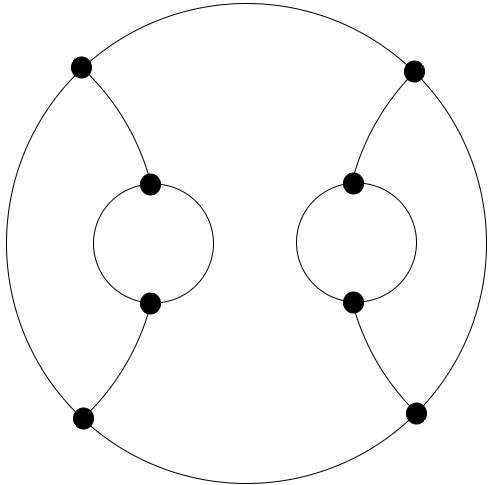} \\ \hline
\centering\arraybackslash$[8,8,2,2,2]$ & \centering\arraybackslash$[8,6,4,2,2,2]$ & \centering\arraybackslash$[8,6,2,2,2,2,2]$ & \centering\arraybackslash$[8,4,4,4,2,2,2]$ \\ \hline
\centering\arraybackslash\includegraphics[width=27mm]{DessinDelta4_5.png} & \centering\arraybackslash\includegraphics[width=27mm]{DessinDelta4_6.png} & 
\centering\arraybackslash\includegraphics[width=27mm]{DessinDelta4_7.png} & \centering\arraybackslash\includegraphics[width=27mm]{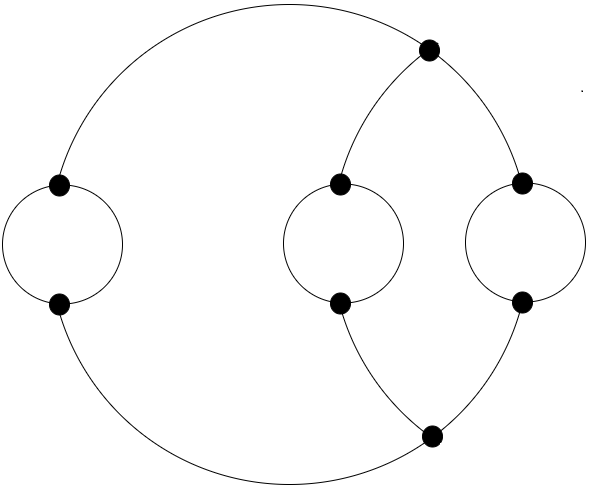} \\ \hline
\centering\arraybackslash$[8,4,4,2,2,2,2]$ & \centering\arraybackslash$[8,4,2,2,2,2,2,2]$ & \centering\arraybackslash$[8,2,2,2,2,2,2,2,2]$ & \centering\arraybackslash$[6,6,6,2,2,2]$ \\ \hline

\centering\arraybackslash\includegraphics[width=27mm]{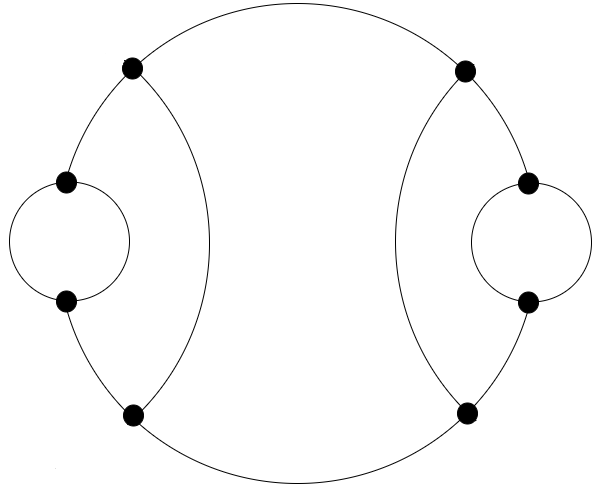} & 
\centering\arraybackslash\includegraphics[width=27mm]{DessinDelta4_10.png} & \centering\arraybackslash\includegraphics[width=26mm]{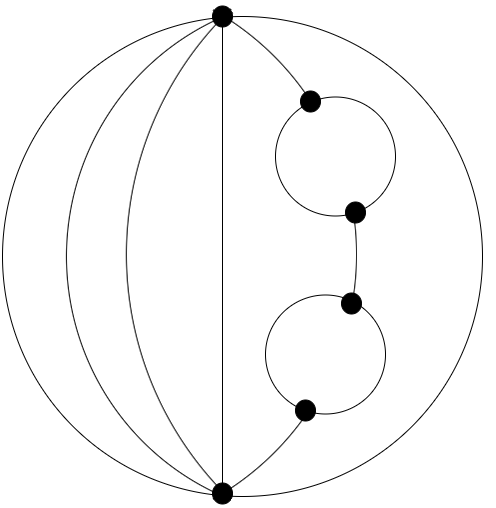} &  
\centering\arraybackslash\includegraphics[width=27mm]{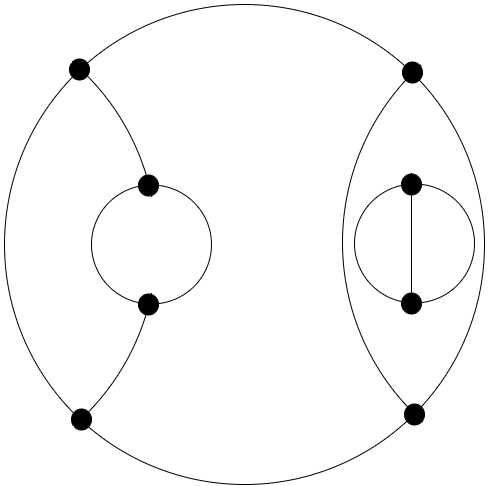} \\  \hline
\centering\arraybackslash$[6,6,4,4,2,2]$ & \centering\arraybackslash$[6,6,4,2,2,2,2]$ & \centering\arraybackslash$[6,6,2,2,2,2,2,2]$ & \centering\arraybackslash$[6,4,4,4,2,2,2]$\\ \hline

\centering\arraybackslash\includegraphics[width=27mm]{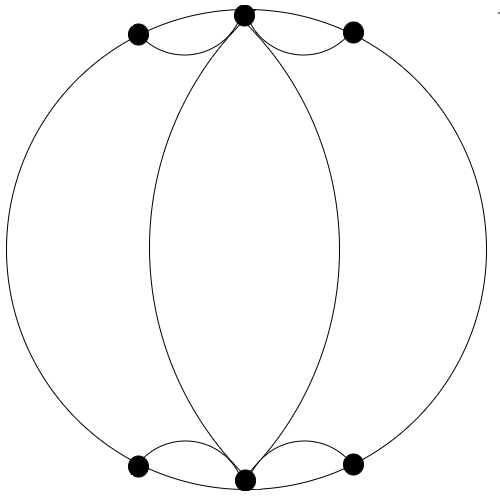} & \centering\arraybackslash\includegraphics[width=26mm]{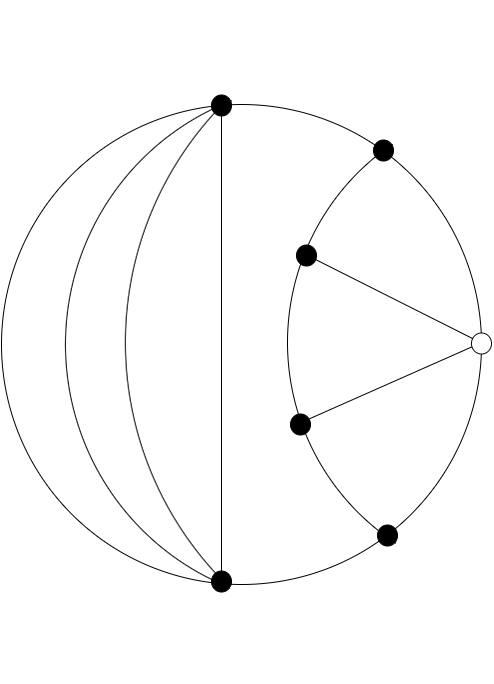} &  
\centering\arraybackslash\includegraphics[width=27mm]{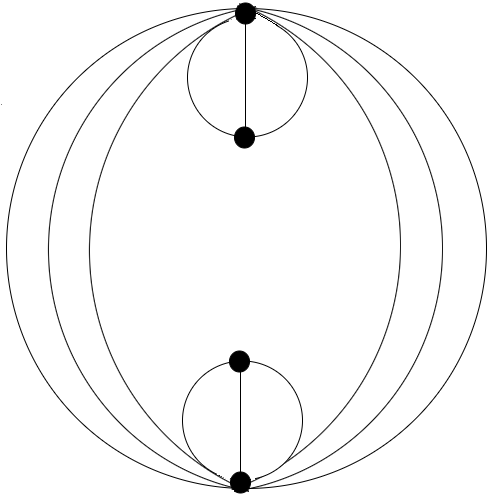} &
\centering\arraybackslash\includegraphics[width=27mm]{DessinDelta4_16.png} \\ \hline
\centering\arraybackslash$[6,4,4,2,2,2,2,2]$ & \centering\arraybackslash$[6,4,2,2,2,2,2,2,2]$ & \centering\arraybackslash$[6,2,\dots,2]$ & \centering\arraybackslash$[4,4,4,4,4,4]$\\ \hline

\centering\arraybackslash\includegraphics[width=26mm]{DessinDelta4_17.png} &  
\centering\arraybackslash\includegraphics[width=27mm]{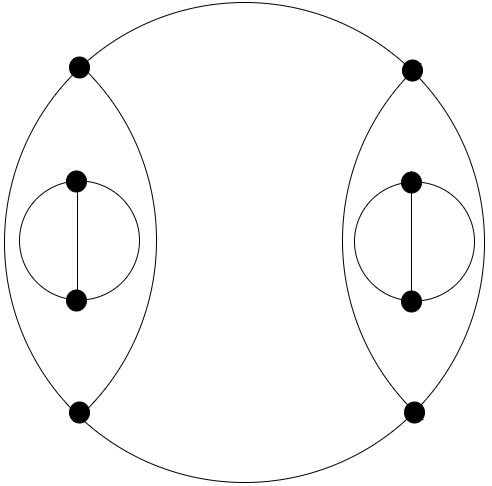} & 
\centering\arraybackslash\includegraphics[width=26mm]{DessinDelta4_19.png} & \centering\arraybackslash\includegraphics[width=27mm]{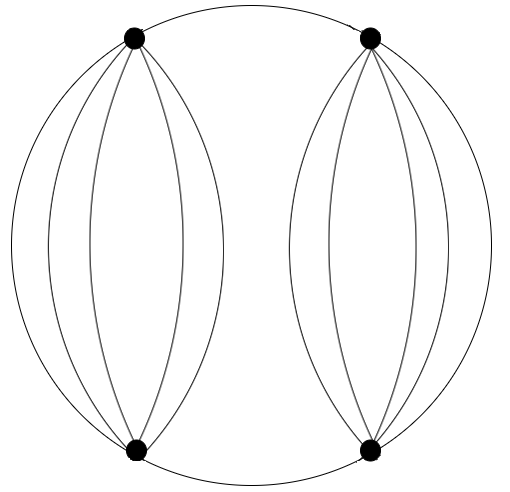} \\ \hline
\centering\arraybackslash$[4,4,4,4,4,2,2]$ & \centering\arraybackslash$[4,4,4,4,2,2,2,2]$ & \centering\arraybackslash$[4,4,4,2,2,2,2,2,2]$& \centering\arraybackslash$[4,4,2,\dots,2]$ \\ \hline
& \centering\arraybackslash\includegraphics[width=27mm]{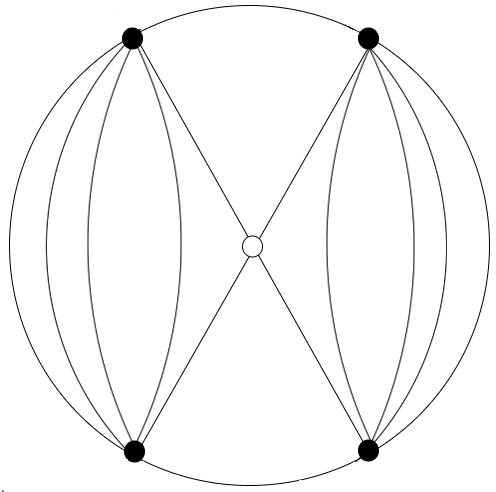} & 
\centering\arraybackslash\includegraphics[width=26mm]{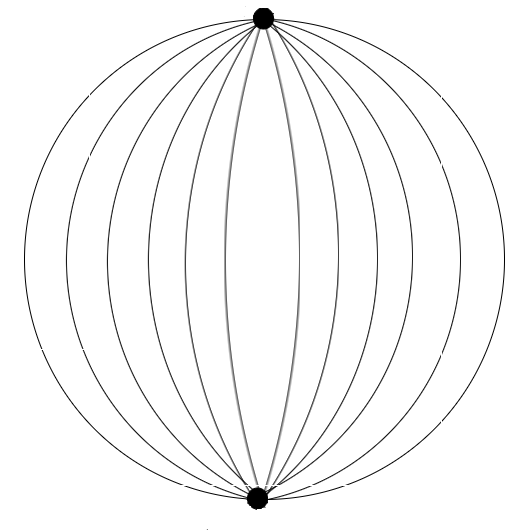}  & \\ \hline

& \centering\arraybackslash$[4,2,\dots,2]$ & \centering\arraybackslash$[2,\dots,2]$ &  \\
\hline

\end{tabular}

\label{table:2}
\end{table}
\section{Deformations of Dessins d'Enfants}

In this section, we present some results on the deformation space of dessins. We list all possible moves that a dessin of a CRTC can have during deformations, define when a dessin is simple and study the role of simple dessins in the deformation space.

Let $C$ be a completely reducible trigonal curve in the form $(y-y_1)(y-y_2)(y-y_3)=0$ where $y_i\in \mathbb{C}[x]$. When any $y_i$ is transformed to another polynomial by changing its coefficients continuously, $\Gamma_C$ also deforms to another dessin d'enfant. 

\begin{definition}\em
A \textit{move} is a transition from one dessin class to another up to the graph isomorphism.
\end{definition}

Here is the proposition about these moves.

\begin{proposition}\em
There are four types of moves in $\Gamma_C$ as follows:
\begin{enumerate}
\item Monochrome modification
\item Merging $\bullet$- vertices
\item Merging $\circ$- vertices
\item Merging $\bullet$- and monochrome vertices
\end{enumerate}

\begin{figure}[h!]
\centering
\includegraphics[width=120mm]{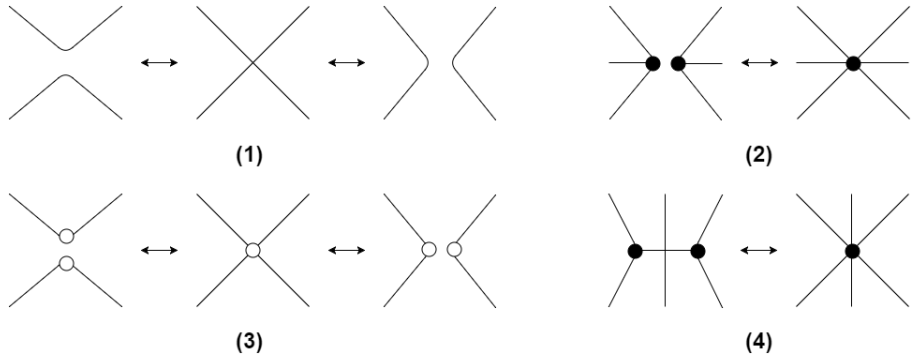}
\caption{Four move types}

\end{figure}

\end{proposition}

\begin{proof}
When $j_C-r=0$ is not irreducible and has a root with multiplicity bigger than 1 for $r\in (0,1)$, we have the monochrome modification. Similarly, when $j_C=0$ and $j_C-1=0$ are not irreducible and have roots with multiplicity bigger than 1, we get the ``merging $\bullet$-" and ``merging $\circ$-" vertices moves respectively. 

Lastly, when there exist a monochrome modification, and the deformation keeps the monochrome vertex and makes $r$ converge to $0$, we get the ``merging $\bullet$- and monochrome" vertices. Note that the deformation cannot keep the monochrome vertex and make $r$ converge to $1$ at the same time since there will be a $\circ$- vertex between two same types of $\bullet$- vertices which can never happen. 

\end{proof}

An example to a merging $\circ$- vertices can be found in Figure \ref{FigElemntary}. The dessin in Figure \ref{FigElemntary}-\textbf{a} is deforming into the dessin in Figure \ref{FigElemntary}-\textbf{c} by merging two $\circ$- vertices as in Figure \ref{FigElemntary}-\textbf{b}.

\begin{figure}[h!]
\centering
\includegraphics[width=100mm]{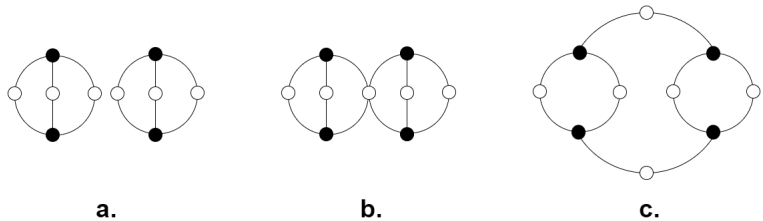}
\caption{An exampe to the merging $\circ$- vertices move}
\label{FigElemntary}
\end{figure}

During the deformations, some dessin d'enfant classes occur more frequently. We will pay more attention to these dessins in the rest of this section. Here is the formal definition of these dessins:

\begin{definition}\em
A dessin d'enfant is called \textit{simple} if it has following properties:
\begin{itemize}
\item Its $\bullet$- vertices are of degree 3,
\item Its $\circ$- vertices are of degree 2,
\item It has no monochrome vertex.
\end{itemize}
In other words, a simple dessin d'enfant is an unbranched covering of the cross-ratio graph. A dessin d'enfant which is not simple is called \textit{multiple}.
\end{definition}

\begin{example}\label{example:simpledessin}\em
For CRTC's that have maximal degree 2, there are only two different simple dessins d'enfants as in Figure \ref{simple2}:

\begin{figure}[ht!]
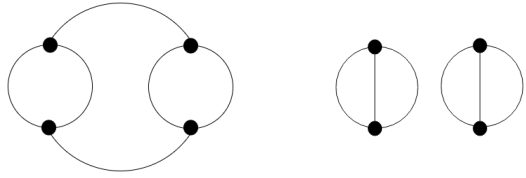


\begin{center}
\begin{tabular}{m{4cm} m{4cm}} 

\includegraphics[width=30mm]{DessinDelta2_2.png} & \includegraphics[width=25mm]{DessinDelta2_1.png} \\ 
\end{tabular}
\end{center}
\caption{The simple dessins d'enfants for maximal degree 2}
\label{simple2}
\end{figure}
\end{example}

As it is mentioned before, simple dessins are observed more frequently during deformations. Furthermore, we also show that they form open subsets in the dessin space $\mathfrak{D}$:

\begin{theorem}\em \label{thm:simple}
Simple dessins d'enfants form open subsets in $\mathfrak{D}$.
\end{theorem}
\begin{proof}
Let $\mathfrak{T}_{d_1,d_2,d_3}$ be the space of completely reducible trigonal curves whose components have degree $d_1,d_2$ and $d_3$. By definition,  $\mathfrak{T}_{d_1,d_2,d_3}\simeq P_1^{d_1+1} \times P_1^{d_2+1} \times P_1^{d_3+1}$. Let $d:=\text{maximum}(d_1,d_2,d_3)$ and $\mathfrak{D}_d$ be the space of all dessins d'enfants for the maximal degree $d$. Let $\omega: \mathfrak{T}_{d_1,d_2,d_3} \rightarrow \mathfrak{D}_d$ be a map that sends each curve to its dessins d'enfants. Let $\mathfrak{J}_{d_1,d_2,d_3}$ be the set of critical values of the $j$-invariant of completely reducible trigonal curves whose components have degree $d_1,d_2$ and $d_3$. By definition, $\mathfrak{J}_{d_1,d_2,d_3}\simeq P_1^d$. Let $\alpha: \mathfrak{T}_{d_1,d_2,d_3} \rightarrow \mathfrak{J}_{d_1,d_2,d_3}$ be another map that sends a trigonal curve to the critical values of its $j$-invariant. Let $\bigtriangleup=\bigcup\limits_{i=1}^d \{c_i | j(c_i) \in [0,1]\}$ where $(c_1,...,c_d) \in \mathfrak{J}_{d_1,d_2,d_3}$ i.e. $\bigtriangleup \subset \mathfrak{J}_{d_1,d_2,d_3}$ is the set of the critical points whose image under the $j$-invariant lies on the interval $[0,1]$. $\alpha^{-1}(\bigtriangleup)$ is the locus of the trigonal curves with multiple dessins d'enfants and codim$(\alpha^{-1}(\bigtriangleup))=1$ in $\mathfrak{T}_{d_1,d_2,d_3}$. Hence, every open set in $\mathfrak{T}_{d_1,d_2,d_3}$ contains a trigonal curve with a simple dessin d'enfant. Since $\omega$ is an open map, simple dessins d'enfants form open subsets in $\mathfrak{D}_d$.
\end{proof}
\begin{example}\em
Let $\mathfrak{T}=\{C:=(y-x^2-a)(y-2x-1)(y+x-1)| a \in P_1\}$ be a one parameter subset of  $\mathfrak{T}_{2,1,1}$. The cross-ratio function is $$cr(x)=\frac{(x^2+a)-(-x+1)}{(2x+1)-(-x+1)} = \frac{x^2+x+a-1}{3x}.$$ The critical values are $x=\pm (a-1)^{1/2}$. Then $\bigtriangleup=\{(a-1)^{1/2}|j(a-1)^{1/2}\in [0,1]\} \cup \{-(a-1)^{1/2}|j(a-1)^{1/2}\in [0,1]\}$ and the picture of $\bigtriangleup$ can be found in Figure \ref{cardio}. The corresponding simple dessins can be seen in Figure \ref{coloredcardio}.

\begin{figure}[ht!]
\centering
\includegraphics[width=110mm]{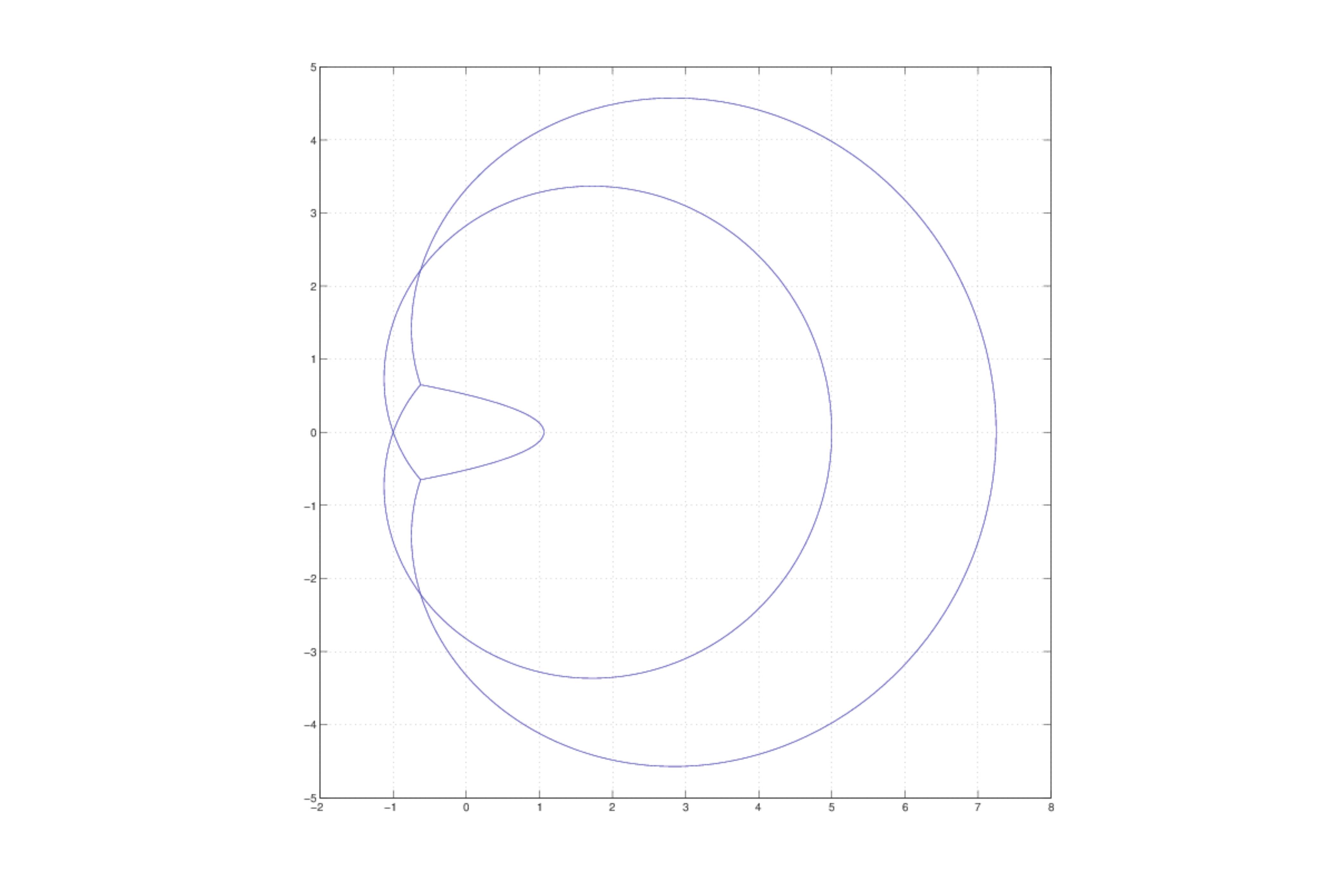}
\caption{The discriminant locus in the one parameter family of trigonal curves $\{C:=(y-x^2-a)(y-2x-1)(y+x-1)| a \in P^1\}$ }
\label{cardio}
\end{figure}

\begin{figure}[ht!]
\centering
\includegraphics[width=120mm]{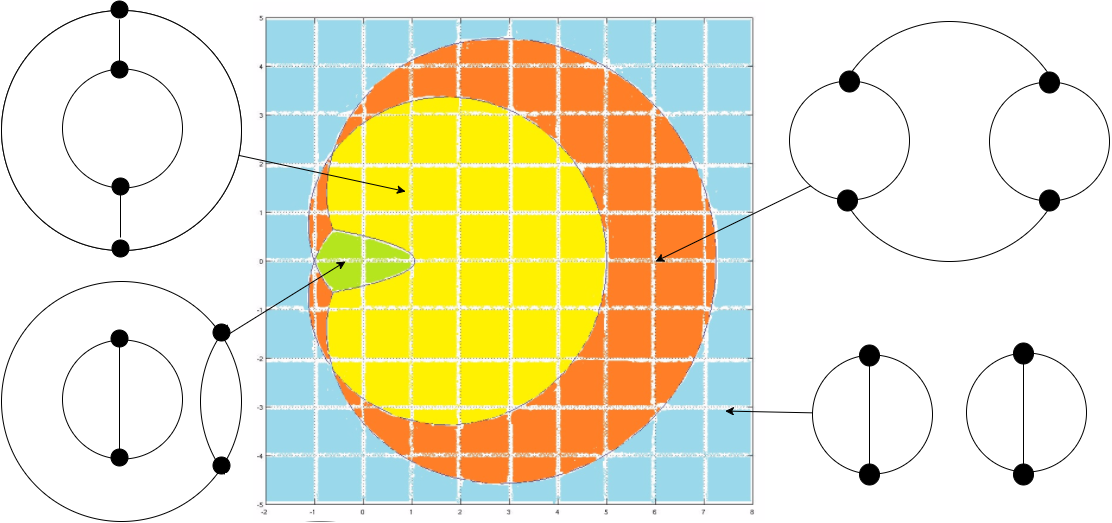}
\caption{The simple dessins in the one parameter family of trigonal curves $\{C:=(y-x^2-a)(y-2x-1)(y+x-1)| a \in P^1\}$ }
\label{coloredcardio}
\end{figure}
\end{example}

\begin{example}\em
Let $\mathfrak{T}=\{C:=(y-x^3-a)(y-2x-1)(y+x-1)| a \in P^1\}$ be a one parameter subset of  $\mathfrak{T}_{3,1,1}$. The discriminant locus of $\mathfrak{T}$ is plotted in Figure \ref{cardio3}
\begin{figure}[ht!]
\centering
\includegraphics[width=120mm]{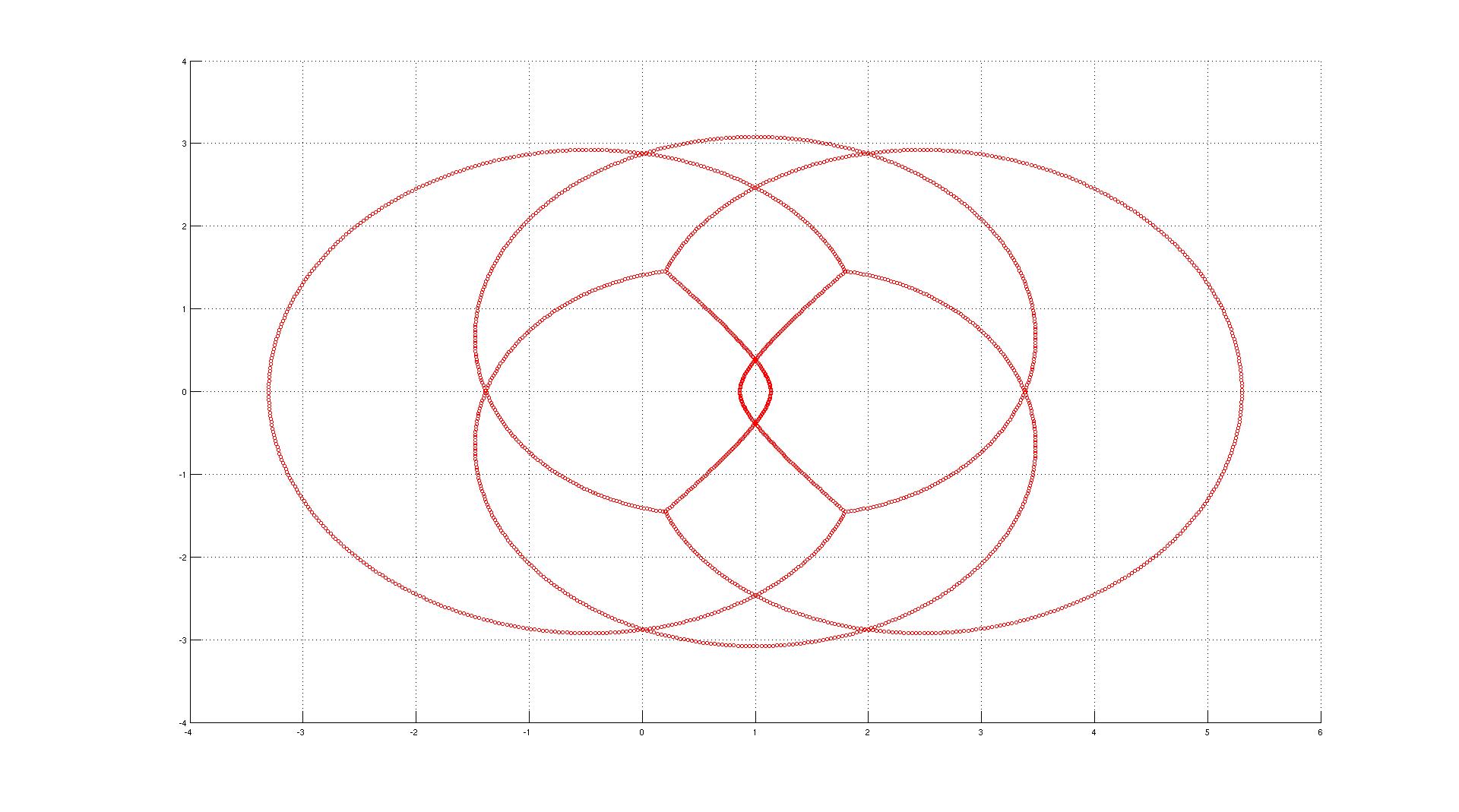}
\caption{The discriminant locus in the one parameter family of trigonal curves $\{C:=(y-x^3-a)(y-2x-1)(y+x-1)| a \in P^1\}$ }
\label{cardio3}
\end{figure}
\end{example}

One can observe from the previous examples that multiple dessins can be obtained by deforming simple dessins:

\begin{corollary}\label{deformdessin}\em
Every dessin d'enfant is the limit of simple dessins d'enfants in the deformation space.
\end{corollary}

\begin{proof}
From Theorem \ref{thm:simple}, multiple dessins d'enfants have codimension 1 in $\mathfrak{D}$ and simple dessins d'enfants form an open subset in $\mathfrak{D}$. In other words, multiple dessins d'enfants form the boundary of the open subsets in $\mathfrak{D}$ hence for a given multiple dessin d'enfant, one can find a series of elementary moves in an open subset that a simple dessin d'enfant deforms to it. 
\end{proof}

The number of simple dessins for a fixed maximal degree is important to both interpret the discriminant locus and also classify the dessins up to the graph isomorphism. Here is an approximation for the number for a fixed maximal degree:

\begin{theorem}\em
Let $u_n$ be the number of simple dessins d'enfants of trigonal curves with the maximal degree $n$. Then $$u_n\sim \frac{(3n)!}{6^{2n}(n!)^2} \exp(2 \minus \frac{2}{9n}+O(n^{\minus 2})).$$

\end{theorem}

\begin{proof}
Since the cases how $\bullet$-vertices connect each other decide the graph type in the simple dessins, it is enough to find out the number of the possible arrangements of connecting $\bullet$-vertices. From Proposition \ref{proposition::valuesofj}, a dessin can have two different $\bullet$-vertices, cyan and yellow. $\Gamma_C$ has $n$ cyan $\bullet$-vertices and $n$ yellow $\bullet$-vertices. We also know that two same type $\bullet$-vertices cannot share an edge and a $\bullet$-vertex has degree 3 by definition.

Under these conditions, one can construct an $n\times n$ adjacency matrix $M$ to find out a simple dessin as follows: Let $c_1,\dots, c_n$ be the cyan $\bullet$-vertices and $y_1,\dots,y_n$ be the yellow $\bullet$-vertices and $m_{i,j}$ be the number of edges between $c_i$ and $y_j$ for some $i,j \in \{1,\dots, n\}$. Hence, $M$ is an $n \times n$ matrix whose sum of rows and columns are equal to 3. Moreover, permuting rows and columns of $M$ gives the same dessin type.

Finally, to get $u_n$, it is enough to find the number of $n\times n$ matrices of non-negative integers with the following restrictions:

1. The sum of each row and column is equal to $3$.

2. Two matrices are considered equal if one can be obtained by permuting rows and/or columns.

For example, for $n=2$ there are two different matrices as follows:

$$M_1=\begin{pmatrix}
1&2\\2&1
\end{pmatrix},\qquad M_2=\begin{pmatrix}
3&0\\0&3
\end{pmatrix}.$$

For $n=3$ there are five different matrices as follows:

$$\begin{pmatrix}
1&1&1\\1&1&1\\1&1&1
\end{pmatrix},\quad \begin{pmatrix}
0&1&2\\1&2&0\\2&0&1
\end{pmatrix}, \quad \begin{pmatrix}
0&1&2\\1&1&1\\2&1&0
\end{pmatrix}, \quad \begin{pmatrix}
1&2&0\\2&1&0\\0&0&3
\end{pmatrix}, \quad \begin{pmatrix}
3&0&0\\0&3&0\\0&0&3
\end{pmatrix}.$$

This problem is asymptotically solved in \cite{greenhill2013asymptotic} where the solution is  $$u_n\sim \frac{(3n)!}{6^{2n}(n!)^2} \exp(2 \minus \frac{2}{9n}+O(n^{\minus 2})).$$
\end{proof}

For small degrees, starting with $n=1$, $u_n$ is 1, 2, 5, 12, 31, 103, 383, 1731, 9273, 57563, 406465. This sequence is OEIS A232215 in the On-Line Encyclopedia of Integer Sequences.

\section{Maximal Dessins d'Enfants}\label{CH:5}

In this section, we define what a maximal dessin d'enfant is and state a sufficient condition that a non-maximal dessin d'enfant needs to have in order to be deformed to a maximal dessin d'enfant. Degtyarev could use only the maximal dessins to compute the braid monodromy and the fundamental group. Hence it is important to know when a non-maximal dessin can be deformed to a maximal one.

We first define the maximal dessins d'enfants.
\begin{definition}\em
A dessin d'enfant is maximal if it is connected, has no monochrome vertex and there is only one $\times$- vertex in each region.
\end{definition}

The number of $\times$- vertices in regions of a given dessin d'enfant is important to be maximal. We state the following proposition about the number of $\times$- vertices inside the regions of a dessin d'enfant:

\begin{proposition}\em
Inside a $2m$ gonal region of $\Gamma_C$ for $m\in \mathbb{Z}^{+}$, there is at least one $\times$- vertex and at most $m$ $\times$- vertices.
\end{proposition}

\begin{proof}
Let $\mathfrak{d}(C)=n$ and $R_{2m}$ be a $2m$-gonal region in $\Gamma_C$. From Remark \ref{remark::rem1}, we know that the $j$-invariant of $C$, $j_C$, factors into two maps as $j_C=\phi_2 \circ \phi_1$ (see Figure \ref{fig:jcons}). Moreover, we have $\phi_2(c)=\infty$ when $c\in\{0,1,\infty\}$ and these three points lie in different regions in the cross-ratio graph (see Figure \ref{crossratio}).

Since $\phi_1$ is a covering, $\phi_1^{-1}$ lifts each region in the cross-ratio graph to regions in the dessin d'enfant. Let fix a region $R$ in the cross-ratio graph and assume that it lifts to $R_{2m}$. Let $\restr{\phi_1}{R_{2m}}:R_{2m} \rightarrow R$ be the restriction $\phi_1$ on $R_{2m}$ and $p\in \{0,1,\infty\}$ be the point in $R$ where $\phi_2(p)=\infty$. Clearly $\restr{\phi_1}{R_{2m}}$ is an $m$-fold covering. If $p$ is not in the ramification locus of $\restr{\phi_1}{R_{2m}}$, there are $m$ $\times$-vertices in $R_{2m}$ which are the lifts of $p$. If it is in the ramification locus, there are less than $m$ $\times$- vertices and if the index is equal to $m$, there is only one $\times$-vertex in $R_{2m}$.
\end{proof}

One of the main question here is whether a non-maximal dessin d'enfant can be maximized by deforming its dessin but fixing the isomorphism type. We first observe that the $\times$- vertices in a region of a dessin d'enfant actually comes from the same equation. For a given CRTC in the form $(y-y_1)(y-y_2)(y-y_3)=0$ where $y_1,y_2,y_3\in \mathbb{C}[x]$, define $S$ as the set of its singular fibers. Clearly, $S=S_{12}\cup S_{23} \cup S_{13}$ where $S_{12},S_{23}$ and $S_{13}$ come from $y_1=y_2$, $y_2=y_3$ and $y_1=y_3$ respectively. 

\begin{proposition}\label{Prop:maxx}\em
By following the same notations as before, fibers in $S_{12},S_{23}$ and $S_{13}$ lie in $BG, RG$ and $RB$ regions of $\Gamma_C$ respectively.
\end{proposition}

\begin{proof}
Each region type in a dessin maps under $\phi_1$ into the same region type in the cross-ratio graph. Hence, the singular fibers in $RB$, $BG$ and $RG$ in a dessin are lifts of $0,1$ and $\infty$ respectively. Remember the definition of the cross-ratio:
$$
\lambda=\frac{y_1-y_3}{y_2-y_3}.
$$ 
Hence, $\lambda=0$ when $y_1=y_3$, so the fibers in $S_{13}$ lie in $RB$ regions. Similarly, $\lambda=1$ when $y_1=y_2$, so the fibers in $S_{12}$ lie in $BG$ regions and $\lambda=\infty$ when $y_2=y_3$, so the fibers in $S_{13}$ lie in $RB$ regions.  
\end{proof}

\begin{proposition}\label{Prop:max}\em
Let $D:=j_C^{-1}(P_{\mathbb{R}}^1)$. Assume $\Gamma_C$ has a region $R$ that has more than one $\times$-vertices in it. If corresponding part of $R$ in $D$ has a monochrome vertex that is connected to all $\times$-vertices, these $\times$-vertices can be merged into one $\times$-vertex. 
\end{proposition}

\begin{proof}
Assume that $R$ has $m$ singular fibers. From Proposition \ref{Prop:maxx}, all $\times$-vertices in $R$ are actually zeros of a polynomial. Hence, it is enough to continuously deform the polynomial that it has only one zero with multiplicity $m$. Since there is a monochrome vertex in $R$ that is connected to all $\times$-vertices, $R$ is the equisingular degeneration of the same region but having one $\times$-vertex. 
\end{proof}

For example, the region in Figure \ref{maxifigure}-I is equisingular degeneration of the region in Figure \ref{maxifigure}-II. However, we do not know whether Figure \ref{maxifigure}-II can be perturbed to Figure \ref{maxifigure}-III.

\begin{figure}[ht!]
\centering
\includegraphics[width=120mm]{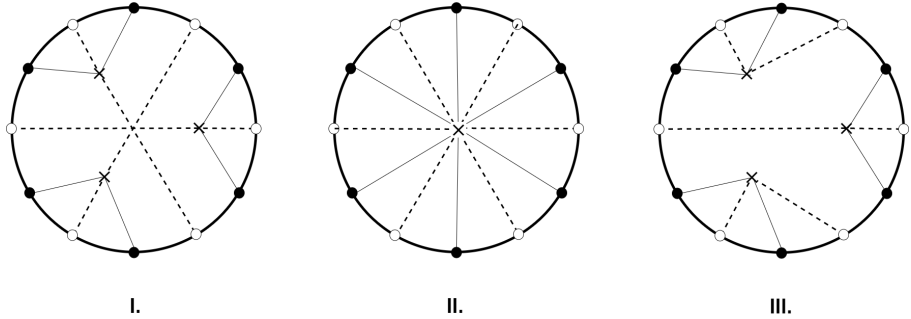}
\caption{Example to maximal and non-maximal regions}
\label{maxifigure}
\end{figure}

\begin{theorem}\em
A non-maximal dessin d'enfant can be deformed into a maximal dessin d'enfant if its regions that have more than one $\times$-vertex satisfy the condition in Proposition \ref{Prop:max}.
\end{theorem}

\begin{proof}
There are two conditions needed to be satisfied to be maximal: having one singular fiber in each region and being connected. The first condition comes from Proposition \ref{Prop:max}. Secondly, assume that the dessin d'enfant is disconnected. To connect two disconnect components, take two nearest $\circ$-vertices, one from each and merge them. Since merging does not change the number or regions and also edges of each region, the dessin d'enfant is still in the same combinatorial class. Hence, one can connect all disconnected components in this way.
\end{proof}

Lastly, we conjecture that one may deform any non-maximal dessin to maximal ones.   
\begin{conjecture}\em
Every non-maximal dessin d'enfant can be deformed to a maximal dessin d'enfant.
\end{conjecture}

\section{Conclusion}

In this paper, we describe some useful properties of the dessin d'enfant of a CRTC, find an upper bound for the number of the combinatorial types and list all possible types for the trigonal curves with a fixed maximal degree. We also give the exact types for the maximal degrees $1,2,3,4$. 

Moreover, we show how dessins deform and the role of simple dessins in the deformation space. We also prove that a non-maximal dessin can be deformed to a maximal one if the regions that have more than one $\times$-vertices have a monochrome vertex that connects them. 

We can list some future works in this area as follows:

(1) We plan to compute the fundamental group and Alexander polynomial of a trigonal curve using its dessin d’enfant. This idea has partially done in \cite{degtyarev2012topology} for the trigonal curves having maximal dessin. Our goal is expanding this idea for any dessin d'enfant.

(2) We want to understand the relation between the combinatorial types and the isomorphism classes. We already have preliminary results for this relation for the maximal degree $2$. 

(3) Another future task is to investigate the deformation space of dessins d'enfants in more detailed. We already developed an algorithm to compute a 1-parameter family of dessins, which can be used to study the deformation space. 

\section*{Acknowledgement}
I would like to thank to E. Hironaka for her encouragement and many fruitful discussions on this research.

\bibliographystyle{plain}
\bibliography{ClassificationDessin.bbl}

\end{document}